\documentclass{article}

\usepackage{geometry}
\usepackage{graphicx} 
\usepackage{epstopdf} 

\usepackage[colorlinks, hypertexnames=false]{hyperref}
\hypersetup{linkcolor=blue, citecolor=red} 

\usepackage{bm}
\usepackage{amsmath} 
\usepackage{amssymb} 
\usepackage{stmaryrd} 
\usepackage{MnSymbol} 
\usepackage{multirow}

\usepackage{amsthm} 
\usepackage{mathtools} 

\usepackage[dvipsnames]{xcolor}
\definecolor{red}{rgb}{1.00,0.00,0.00}
{\numberwithin{equation}{section}
\setlength{\parindent}{1em}

\newtheorem{theorem}{Theorem}[section]
\newtheorem{lemma}{Lemma}[section]
\newtheorem{remark}{Remark}[section]

\renewcommand{\footnotesize}{\scriptsize}
\newcommand{\normmm}[1]{{\left\vert\kern-0.25ex\left\vert
\kern-0.25ex\left\vert #1
    \right\vert\kern-0.25ex\right\vert\kern-0.25ex\right\vert}}

\geometry{left=3cm,right=3cm,top=4cm,bottom=2.5cm}

\begin{document}           

\title{A pressure robust staggered discontinuous Galerkin method \\for the Stokes equations}
\author{Lina Zhao\footnotemark[1]\qquad
  \;Eun-Jae Park\footnotemark[2]\qquad
    \;Eric Chung\footnotemark[3]}
\renewcommand{\thefootnote}{\fnsymbol{footnote}}
\footnotetext[1]{Department of Mathematics,The Chinese University of Hong Kong, Hong Kong SAR, China. ({lzhao@math.cuhk.edu.hk}).}
\footnotetext[2]{Department of Computational Science and Engineering, Yonsei University, Seoul 03722, Republic of Korea. ({ejpark@yonsei.ac.kr}).}
\footnotetext[3]{Department of Mathematics, The Chinese University of Hong Kong, Hong Kong SAR, China. ({tschung@math.cuhk.edu.hk}).}

\maketitle

\textbf{Abstract:}
In this paper we propose a pressure robust staggered discontinuous Galerkin method for the Stokes equations on general polygonal meshes by using piecewise constant approximations. We modify the right hand side of the body force in the discrete formulation by exploiting divergence preserving velocity reconstruction operator, which is the crux for pressure independent velocity error estimates. The optimal convergence for velocity gradient, velocity and pressure are proved. In addition, we are able to prove the superconvergence of velocity approximation by the incorporation of divergence preserving velocity reconstruction operator in the dual problem, which is also an important contribution of this paper. Finally, several numerical experiments are carried out to confirm the theoretical findings.



\pagestyle{myheadings} \thispagestyle{plain} \markboth{LinaEric}
    {PR-SDG method for the Stokes equations}


\section{Introduction}

In this paper we consider the following Stokes equations: Find the velocity $\bm{u}:\Omega\rightarrow \mathbb{R}^2$ and the pressure $p:\Omega\rightarrow \mathbb{R}$ such that
\begin{equation}
\begin{split}
-\nu\Delta \bm{u}+\nabla p&=\bm{f} \hspace{0.3cm} \mbox{in}\; \Omega, \\
\nabla\cdot \bm{u}&=0 \quad \mbox{in} \; \Omega,\\
\bm{u}&=0 \quad \mbox{on} \; \partial \Omega,
\end{split}
\label{mixed}
\end{equation}
where $\Omega$ is the computational domain in $\mathbb{R}^2$, $\nu>0$ is a constant viscosity parameter and $\bm{f}\in [L^2(\Omega)]^2$ is the body force.

A large amount of work has been dedicated to solving \eqref{mixed}, see \cite{Girault86,Layton08} and the references therein. However, classical finite element methods can not deliver exactly divergence free solutions in the sense of $H(\text{div};\Omega)$ and the divergence constraint is relaxed in order to guarantee inf-sup stability \cite{JohnLinke17}. As such, these methods suffer from a lack of pressure robustness, specifically, their velocity error estimates are pressure dependent and possibly deteriorates unboundedly for $\nu\rightarrow 0$, which shows some kind of locking phenomena in the sense of \cite{BabuskaSuri92}.


To remedy this issue, various approaches have been proposed. Exactly divergence-free $H^1(\Omega)$ or $H(\text{div};\Omega)$ conforming methods of order $k$ is developed in \cite{ScottVogelius85,Zhang11,FalkNeilan13}. An alternative strategy to achieve divergence-free property is to enrich the $H(\text{div};\Omega)$-conforming elements with divergence free rational functions, which provides the correct flexibility to enforce (strong) continuity (cf. \cite{GuzmanNeilan14,GuzmanNeilan143D}). Moreover, Cockburn \cite{CockburnKanschat07} and Wang \cite{WangYe07} modify the variational formulation and introduce the tangential penalty and thus obtain discontinuous Galerkin divergence-free schemes. Another approach that can deliver pressure independent error estimates is to add grad-div stabilization \cite{Olshanskii04,OlshanskiiLube09}. Recently, a method based on modifying the right hand side is shown to be able to retrieve pressure robustness, and this strategy has been successfully applied to finite volume method \cite{Linke12}, the nonconforming Crouzeix-Raviart element \cite{Linke14,Brennecke15}, the conforming $P_2^+-P_1^{\text{disc}}$ element \cite{LinkeMatthies16}, the hybrid discontinuous Galerkin method \cite{PietroErn16}, low and high order Taylor-Hood and mini elements \cite{LedererLinke17}, hybrid high-order (HHO) method \cite{QuirozPietro20} and weak Galerkin (WG) method \cite{Mu20}.

Recently, polygonal finite element methods have become a hot topic due to their great flexibility in handling complicated geometries. Various approaches in the framework of polygonal grids have been proposed to solve partial differential equations arising from practical applications, such as virtual element method (VEM) \cite{veiga2013,veiga2014}, mimetic finite difference methods \cite{Lipnikov14}, HHO method \cite{pietro2014,pietro2020}, WG methods \cite{WangYe13,WangYe14}, generalized barycentric coordinates methods \cite{sukumar2004}, staggered discontinuous Galerkin (SDG) method \cite{LinaPark18,LinaParkShin19}, etc. SDG as the new generation discretisation methods for PDEs based on discrete unknowns that enjoys staggered continuity properties is initially introduced to solve wave propagation problem \cite{EricEngquistwave06,EricEngquistwave09}. Its close connection to a hybridizable DG method is unveiled in \cite{ChungCockburnFu14,ChungCockburnFu16}. Note that SDG method differs from other DG methods in the sense that the basis functions are locally conforming and penalty terms are not required. This method has many desirable features and has been successfully applied to a wide range of partial differential equations \cite{Cheung15,Du18,LinaPark18,LinaParkShin19,LinaParkcoupledSD20,
LinaParkhybrid20,LinaChungLam20,LinaParkelasticity20,KimLinaPark20}.

The purpose of this paper is to develop a pressure robust staggered discontinuous Galerkin (PR-SDG) method on general polygonal meshes for the Stokes equations with piecewise constant approximations. Our approach is based on the framework of lowest order SDG method introduced in \cite{LinaParkShin19}. The novelty here is to exploit velocity reconstruction operator in the discretization of the source term in the spirit of \cite{Linke12,Mu20}. The principal idea behind this is that discrete divergence-free velocity test functions are mapped to exact divergence-free ones by velocity reconstruction operator. The crux of SDG method for the Stokes equations is to generate a sub-triangulation by connecting an interior point of the polygonal grid to all the vertices of the polygon. Then finite element spaces that enjoy staggered continuity properties for velocity gradient, velocity and pressure are developed, which naturally lead to inf-sup stable pairs. In particular, the velocity space is continuous over the edges of the polygon, which enables us to establish velocity reconstruction operator based on the polygonal grid. A rigorous convergence analysis for $L^2$ errors of velocity gradient, velocity and pressure is investigated. The main difficulty lies in the proof of superconvergence and traditional techniques no long work. To attack this issue, we exploit Aubin-Nitsche duality argument with nonstandard incorporation of continuous and discrete formulation of the dual problem; in addition, velocity reconstruction operator is employed in the discretization of the dual problem. We keep track explicitly of the dependence on the viscosity in the analysis so as to address the practically important issue of body forces with large irrotational part. Note that the lowest order SDG method can be viewed as finite volume method, thus our work provides new perspectives for the understanding of finite volume method. Indeed, our primal and dual partitions are exactly the same as the finite volume method proposed in \cite{Ye06}, whereas our method is based on the first order system and piecewise constant functions are exploited for all the involved unknowns. Importantly, the continuity of all the unknowns are staggered on the interelement boundaries, which naturally yields a stable numerical scheme without resorting to further stabilization. We emphasize that the development of pressure robust method with velocity reconstruction operator on polygonal mesh is still in its infancy \cite{Chen19,Mu20}, the approach developed in this paper will definitely inspire more works in this direction.

The rest of the paper is organized as follows. In the next section, we formulate the PR-SDG method for the Stokes equations based on velocity reconstruction operator. Then in section~\ref{sec:analysis}, we prove the optimal convergence for $L^2$ errors of velocity gradient, velocity and pressure as well as superconvergence of velocity. Several numerical experiments are carried out in section~\ref{sec:numerical} to verify the proposed theories. Finally, a conclusion is given.

\section{Description of PR-SDG method}
In this section we begin with introducing the construction of minimal degree $H(\text{div};\Omega)$ conforming function on convex polygon. Then we propose the PR-SDG method by employing $H(\text{div};\Omega)$ conforming velocity reconstruction in the discretization of the body force. Finally, some fundamental properties inherited from the proposed method are provided.

\subsection{Preliminaries}

Letting $L^2_0(\Omega):=\{q\in L^2(\Omega): \int_{\Omega}q\;dx=0\}$,
the weak formulation for \eqref{mixed} reads as follows: Find $(\bm{u},p)\in [H^1_0(\Omega)]^2 \times L^2_0(\Omega)$ such that
\begin{equation}
\begin{split}
a(\bm{u},\bm{v})+b(\bm{v}, p)&=(\bm{f},\bm{v}) \quad \forall \bm{v}\in [H^1_0(\Omega)]^2,\\
b(\bm{u}, q)&=0 \hskip+1.1cm \forall q\in L^2_0(\Omega),
\end{split}
\label{weak}
\end{equation}
where
\begin{align*}
a(\bm{u},\bm{v})=(\nu\nabla \bm{u},\nabla \bm{v}),\quad b(\bm{v}, q)&=-(q,\nabla \cdot\bm{v}).
\end{align*}
The weak formulation \eqref{weak} is well posed thanks to the coercivity of the bilinear form $a$ and the inf–sup stability of the bilinear $b$, see \cite{Girault86}.

We introduce an auxiliary unknown $\bm{\omega} =\nu \nabla \bm{u}$. Then, the model problem \eqref{mixed} can be recast into the following first order system of equations
\begin{subequations}\label{firstorder}
\begin{align}
\bm{\omega} -\nu \nabla \bm{u} & = \bm{0} \hspace{0.4cm} \mbox{in}\; \Omega,\label{eq:firstorder1}\\
-\nabla\cdot\bm{\omega}+ \nabla p &=\bm{f} \quad \mbox{in}\; \Omega,\label{eq:firstorder2}\\
\nabla\cdot \bm{u}&=0 \hspace{0.4cm} \mbox{in}\; \Omega,\label{eq:firstorder3}\\
\bm{u}&=\bm{0} \hspace{0.4cm} \mbox{on} \; \partial \Omega.
\end{align}
\end{subequations}

Before closing this subsection we introduce some notations that will be employed throughout the paper.
Let $D\subset \mathbb{R}^d,$ $d=1,2$, we adopt the standard notations for the Sobolev spaces $H^s(D)$ and their associated norms $\|\cdot\|_{s,D}$, and semi-norms $|\cdot|_{s,D}$ for $s\geq 0$.
The space $H^0(D)$ coincides with $L^2(D)$, for which the norm is denoted as $\|\cdot\|_{D}$.
We use $(\cdot,\cdot)_D$ to denote the inner product for $d=1,2$.
If $D=\Omega$, the subscript $\Omega$ will be dropped unless otherwise mentioned.
In the sequel, we use $C$ to denote a generic positive constant which may have different values at different occurrences.

\subsection{\texorpdfstring{$H(\text{div})$}- conforming function on polygon}

In this subsection, we briefly introduce the construction of $H(\text{div};\Omega)$ conforming function on convex polygon by following \cite{Chen17}.
Let $T$ be a polygon with $m$ vertices $\mathsf{v}_i$ that are arranged counterclockwise, and let $\bm{n}_i (1\leq i\leq m)$ be the outward unit normal vector on edge $e_i$ that connects vertices
$\mathsf{v}_i$ and $\mathsf{v}_{i+1}$,
where we conveniently denote $\mathsf{v}_j=\mathsf{v}_{j(mod\, m)}$ when the subscript $j$ is not in the range of $\{1,\cdots,m\}$, see Figure~\ref{Wachspress} for an illustration. Let $\mathsf{x}$ be an interior point of $T$, its distance to edge $e_i$ and a scaled normal vector are defined as
\begin{align*}
d_i = (\mathsf{v}_i-\mathsf{x})\cdot\bm{n}_i,\quad \tilde{\bm{n}}_i = \frac{1}{d_i}\bm{n}_i\quad 1\leq i\leq m.
\end{align*}
Then the Wachspress coordinates are defined as (cf. \cite{Wachspress75})
\begin{align*}
\lambda_i = w_i(x)/W(x)\quad 1\leq i\leq m,
\end{align*}
where $w_i(x)=\text{det}(\tilde{\bm{n}}_i,\tilde{\bm{n}}_{i+1})$, $W(x) = \sum_{i=1}^m w_i(x)$.

\begin{figure}[t]
\centering
\includegraphics[width =0.3\textwidth]{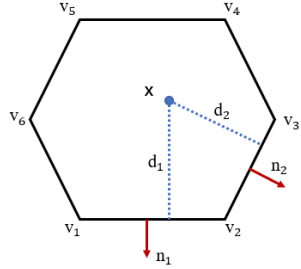}
\caption{Schematic of Wachspress coordinates on a hexagon.}
\label{Wachspress}
\end{figure}

As discusses in \cite{Floater14}, one introduces an auxiliary ratio function
\begin{align*}
R_i(x) = \frac{\nabla w_i(x)}{w_i(x)}\quad 1\leq i\leq m.
\end{align*}
By the quotient rule for differentiation, one is led to
\begin{align*}
\nabla \lambda_i(x) = \lambda_i(x)(R_i-\sum_{j=1}^m \lambda_jR_j)\quad 1\leq i\leq m.
\end{align*}
Then, a counterclockwise $90^o$ rotation of the gradient is defined by
\begin{align*}
\text{curl}(\lambda_i) = \left(
                           \begin{array}{c}
                             -\partial_y \lambda_i \\
                             \partial_x \lambda_i \\
                           \end{array}
                         \right)=\left(
                                 \begin{array}{cc}
                                       0 & -1 \\
                                       1 & 0 \\
                                  \end{array}
                                 \right)
                                 \nabla \lambda_i(x)\quad 1\leq i\leq m.
\end{align*}

Now we are ready to present the construction of minimal degree $H(\text{div};\Omega)$ conforming finite element function on convex polygon. Denote $|T|, |e_i|$ as the area and the length of edge $e_i$ for $1\leq i\leq m$. Let $\mathsf{x}_*$ be an arbitrary point inside polygon $T$, and denote by $|T_i|$ the area of the triangle formed by $\mathsf{x}_*, \mathsf{v}_i, \mathsf{v}_{i+1}$. Then for each $1\leq i\leq m$, we define $\bm{\varphi}_i$ by
\begin{align*}
\bm{\varphi}_i = c_{i,0}(\mathsf{x}-\mathsf{x}_*)+\sum_{k=1}^m c_{i,k} \text{curl}\lambda_k,
\end{align*}
where $c_{i,0}=\frac{|e_i|}{2|T|}$, $c_{i,k}=-\frac{1}{m}\sum_{l=1}^{m-1}lb_{i,k+l}$ and $b_{i,l}=\delta_{i,l}|e_l|-|e_i|\frac{|T_l|}{|T|}$. Here $\delta_{i,l}$ is the Kronecker symbol.

Furthermore, for these basis functions, we have the following properties
\begin{align*}
\bm{\varphi}_i|e_j\cdot\bm{n}_j = \delta_{i,j}\quad \nabla\cdot \bm{\varphi}_i = 2c_{i,0}\quad \forall 1\leq i,j\leq m.
\end{align*}

\begin{lemma}\label{lemma:Hdiv}

For $1\leq i\leq m$, one has \cite{Chen17}
\begin{align*}
\|\bm{\varphi}_i\|_{0,T}\leq C(m)|e_i|.
\end{align*}
In addition, for $\bm{\xi}\in [H^1(T)]^2$, we have
\begin{align*}
\|\bm{\xi}-\Pi^{RT}\bm{\xi}\|_{0,T}\leq C(m)h_T\|\bm{\xi}\|_{1,T}.
\end{align*}

Let $\Pi^0$ be the nodal value interpolation operator based on the generalized barycentric coordinates such as Wachspress coordinates (cf. \cite{Wachspress75}). For any $\phi\in W^{1,p}(T), p>2$, it holds (cf. \cite{Chen17})
\begin{align}
\Pi^{RT}\textnormal{curl}\,\phi=\textnormal{curl}\,\Pi^0\phi.\label{eq:Ih}
\end{align}
\end{lemma}

\subsection{Construction of PR-SDG method}

In this subsection we will present the construction of PR-SDG method. To begin with, we will describe the construction of the staggered mesh needed in our method by following \cite{EricEngquistwave06,EricEngquistwave09,ChungLee12,KimChung13}. We let $\mathcal{T}_{u}$ be the
initial partition of the domain $\Omega$ into non-overlapping simple convex polygonal elements (primal mesh).
We use $\mathcal{F}_{u}$ to represent the set of all the edges in this partition (primal edges) and $\mathcal{F}_{u}^{0}$ to represent the
subset of all interior edges, that is, 
$\mathcal{F}_{u}^{0} = \mathcal{F}_{u}\setminus \partial\Omega$.
For each polygon $T$ in the initial partition $\mathcal{T}_u$, we select an interior point $\mathsf{x}_*$ and
 create new edges by
connecting $\mathsf{x}_*$ to the vertices of polygon.
This process will divide $T$ into the union of triangles, where the triangle is denoted as $\tau$.
Moreover, we will use $\mathcal{F}_{p}$ to denote the set of all the new edges generated by this subdivision process (dual edges) and
use $\mathcal{T}_{h}$ to denote the resulting triangulation (simplicial submeshes),
on which our basis functions are defined.
In addition, we define $\mathcal{F}:=\mathcal{F}_{u}\cup \mathcal{F}_{p}$ and $\mathcal{F}^{0}:=\mathcal{F}_{u}^{0}\cup \mathcal{F}_{p}$.
This construction is illustrated in Figure~\ref{fig:mesh},
where solid lines are edges in $\mathcal{F}_u$
and dashed lines are edges in $\mathcal{F}_p$. For each triangle
$\tau\in \mathcal{T}_h$, we let $h_\tau$ be the diameter of
$\tau$ and $h=\max\{h_\tau, \tau\in \mathcal{T}_h\}$. Here we assume that the primal partition satisfies the standard mesh regularity assumptions: 1) Every primal element $T$ is star-shaped with respect to a ball of radius $\rho_B h_T$, where $\rho_B$ is a positive constant. 2) For every primal element $T$ and every edge $e\in \partial T$, it satisfies $|e|\geq \rho_E h_T$, where $\rho_E$ is a positive constant. Note that these assumptions can guarantee that the resulting triangulation $\mathcal{T}_h$ is shape regular.

For each interior edge $e\in \mathcal{F}_{u}^0$, we use $D(e)$ to denote the dual mesh, which is the union of the two triangles in $\mathcal{T}_h$ sharing the edge $e$,
and for each boundary edge $e\in\mathcal{F}_u\backslash\mathcal{F}_u^0$, we use $D(e)$ to denote the triangle in $\mathcal{T}_h$ having the edge $e$,
see Figure~\ref{fig:mesh}. We define a unit normal vector $\bm{n}_e$ on each edge $e\in \mathcal{F}$ as follows: If $e\in \mathcal{F}\backslash \mathcal{F}^0$ is a boundary edge, then we define $\bm{n}_e$ as the unit normal vector of $e$ pointing towards outside of $\Omega$. If $e\in \mathcal{F}^0$ is an interior edge, then we fix $\bm{n}_e$ as one of the two possible unit normal vectors on $e$. We will use $\bm{n}$ instead of $\bm{n}_e$ to simplify the notation when there is no confusion.

\begin{figure}[t]
\centering
\includegraphics[width =0.35\textwidth]{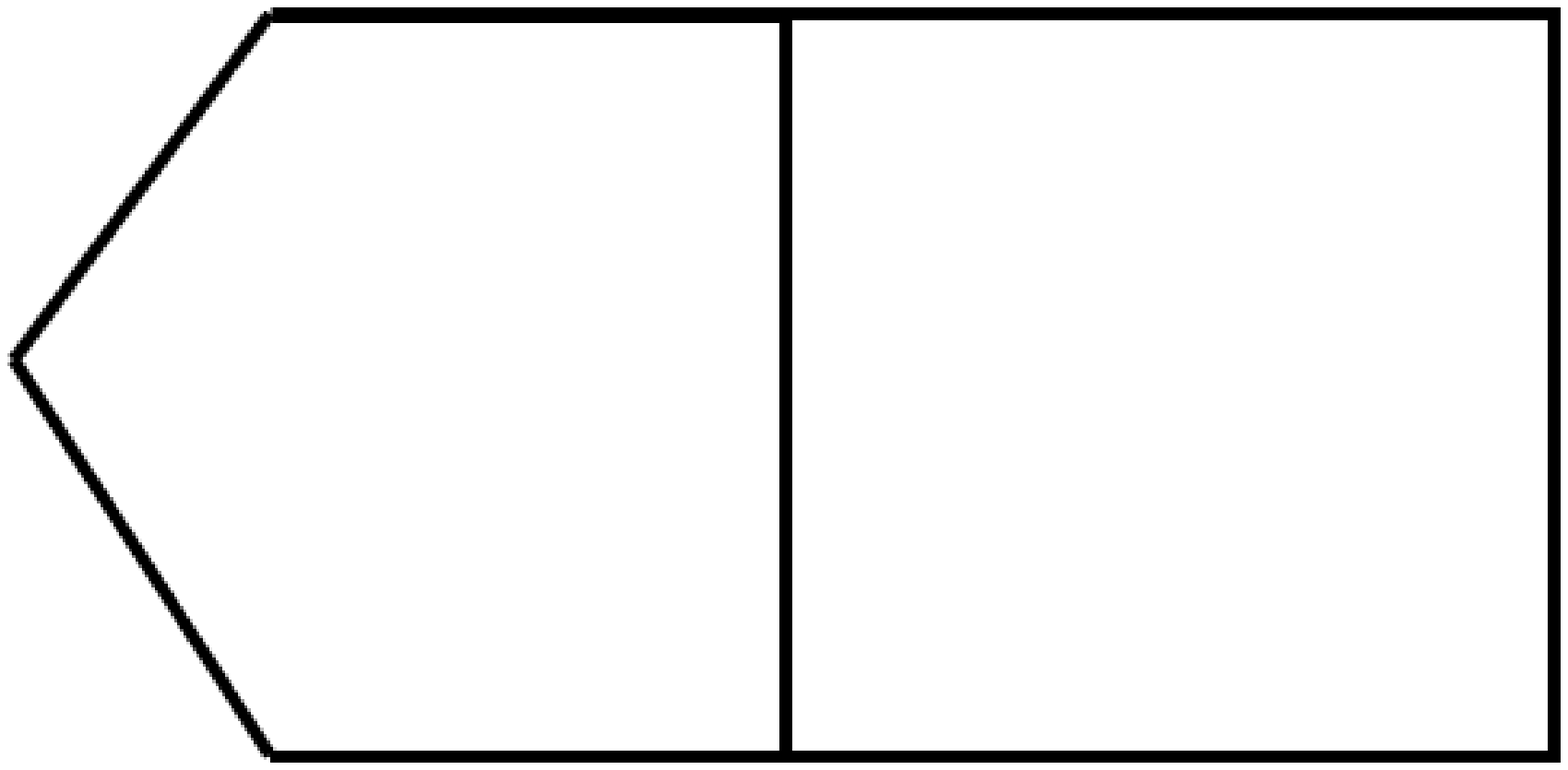}
\includegraphics[width =0.35\textwidth]{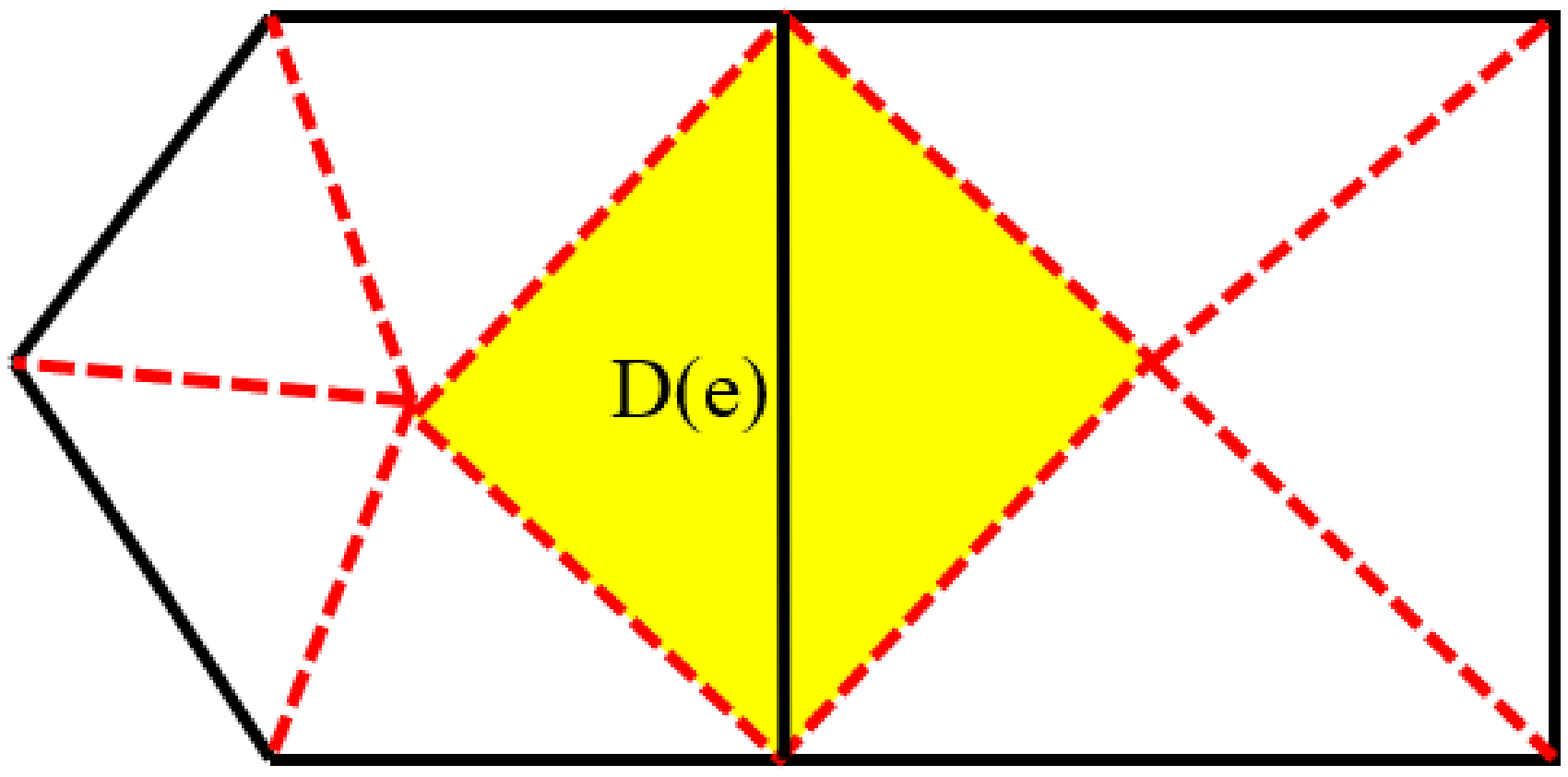}
\caption{Illustration of primal mesh (left) and the resulting simplicial submeshes (right). Solid lines represent primal edges and dashed lines represent dual edges.}
\label{fig:mesh}
\end{figure}

Let $k\geq 0$ be the order of approximation. For every $\tau
\in \mathcal{T}_{h}$ and $e\in\mathcal{F}$,
we define $P_{k}(\tau)$ and $P_{k}(e)$ as the spaces of polynomials of degree less than or equal to $k$ on $\tau$ and $e$, respectively.
Now we are ready to describe the finite element spaces that will be used to define our numerical scheme. First, the locally $H^1(\Omega)$ conforming SDG space for velocity is defined as:
\begin{equation*}
S_{h}:=\{\bm{v} : \bm{v}\mid_{D(e)}\in [P_{0}(D(e))]^2, \forall e \in \mathcal{F}_u; \bm{v}\mid_{D(e)}=0\;\mbox{if}\; e\in\mathcal{F}_u\backslash \mathcal{F}_u^0\}.
\end{equation*}
The degrees of freedom for this space can be described as (see Figure~\ref{dof})
\begin{align*}
\phi_e(\bm{v}):=(\bm{v},\bm{\varsigma})_e\quad \forall \bm{\varsigma}\in [P_0(e)]^2,e\in \mathcal{F}_u.
\end{align*}
The discrete $L^2$ norm and $H^1$ norm for the space $S_h$ are given by
\begin{equation*}
\begin{split}
\|\bm{v}\|_{X}^2&=\|\bm{v}\|_{0}^2+\sum_{e\in \mathcal{F}_u^0}h_e\|\bm{v}\|_{0,e}^2,\\
\|\bm{v}\|_{h}^2& =\sum_{e\in \mathcal{F}_p}h_e^{-1}\|[\bm{v}]\|_{0,e}^2.
\end{split}
\end{equation*}
Here $\bm{v}_i = \bm{v}|_{\tau_i}, i=1,2$ and $[\bm{v}]=\bm{v}_{1}-\bm{v}_{2}$ denotes the jump on an interior edge that is shared by two triangles $\tau_1$ and $\tau_2$ belonging to $\mathcal{T}_h$, and we simply take $[\bm{v}]=\bm{v}_{1}$ for $e\in \mathcal{F}_u\backslash \mathcal{F}_u^0$.

Next, the locally $H(\mbox{div};\Omega)$ conforming
SDG space for the velocity gradient approximation is defined as:
\begin{equation*}
V_{h}:=\{\bm{\psi}: \bm{\psi}\mid_{\tau} \in [P_{0}(\tau)]^{2\times 2},\forall \tau \in \mathcal{T}_{h};[\bm{\psi}\bm{n}]\mid_e=0, \forall e\in \mathcal{F}_{p}\},
\end{equation*}
which is equipped by
\begin{align*}
\|\bm{\psi}\|_{X'}^2=\|\bm{\psi}\|_0^2+\sum_{e\in \mathcal{F}_p}h_e\|\bm{\psi}\bm{n}\|_{0,e}^2.
\end{align*}
Invoking scaling arguments, we have
\begin{align}
\|\bm{\psi}\|_0\leq \|\bm{\psi}\|_{X'}\leq C\|\bm{\psi}\|_0.\label{eq:scaling-psi}
\end{align}
We define the following degrees of freedom for $V_h$ and it is illustrated in Figure~\ref{dof}
\begin{align*}
\varphi_e(\bm{\psi}):=(\bm{\psi}\bm{n},\bm{\varsigma})_e \quad \forall \bm{\varsigma}\in [P_0(e)]^2,e\in \mathcal{F}_p.
\end{align*}

In the above definition, the jump $[\bm{\psi}\bm{n}]$ over an edge $e\in \mathcal{F}_p$ is defined as
\begin{align*}
[\bm{\psi}\bm{n}]=\bm{\psi}_1\bm{n}-\bm{\psi}_2\bm{n},
\end{align*}
where $\bm{\psi}_i=\bm{\psi}|_{\tau_i}$, $e$ is the common edge of the two triangles $\tau_1$ and $\tau_2$ that belong to $\mathcal{T}_h$, and $\bm{n}$ is a unit normal to the edge $e$.
\begin{figure}[t]
\centering
\includegraphics[width =0.7\textwidth]{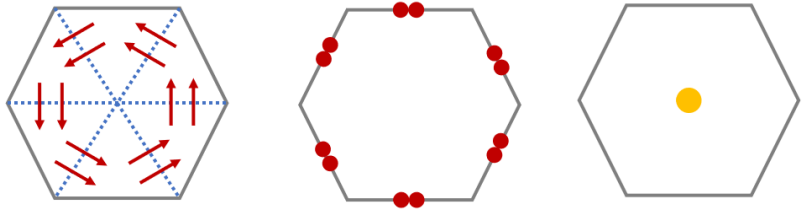}
\caption{Schematic of degrees of freedom for $V_h$ (left), $S_h$ (middle) and $P_h$ (right) over the polygon.}
\label{dof}
\end{figure}

Finally, locally $H^{1}(\Omega)$ conforming finite element space for pressure is defined as:
\begin{equation*}
P_{h}:=\{q: q\mid_{T}\in P_{0}(T),\;\forall T\in \mathcal{T}_u;\int_{\Omega}q\;dx=0\}
\end{equation*}
with norm
\begin{align*}
\|q\|_P^2=\|q\|_0^2+\sum_{e\in \mathcal{F}_p}h_e\|q\|_{0,e}^2.
\end{align*}

Now let us define the $H(\text{div};\Omega)$ conforming interpolation operator on polygonal mesh in the spirit of \cite{Chen17}. For any $\bm{v}\in H(\text{div};\Omega)+S_h$, we define $\Pi^{RT}\bm{v}$ restricted to $T$ by
%
\begin{align*}
\Pi^{RT}\bm{v}=\sum_{i=1}^m r_i(\bm{v})\bm{\varphi}_i,
\end{align*}
where
\begin{align*}
r_i(\bm{v})=
\frac{1}{|e_i|}\int_{e_i} \bm{v}\cdot\bm{n}\;ds\quad \forall e\in \mathcal{F}_u\cap \partial T, T\in \mathcal{T}_u.
\end{align*}
Note that any function $\bm{v}\in S_h$ is continuous over the edge $e\in \mathcal{F}_u$, which by definition yields a function $\Pi^{RT}\bm{v}$ that belongs to $H(\text{div};\Omega)$.

Then following \cite{LinaParkShin19}, the discrete formulation for the Stokes equations \eqref{firstorder} reads as follows: Find $(\bm{\omega}_h, \bm{u}_h,p_h)\in V_h\times S_h\times P_h$ such that
\begin{subequations}\label{eq:SDG}
\begin{align}
B_h(\bm{\omega}_h, \bm{v})+b_h^*(p_h, \bm{v})&=(\bm{f},\Pi^{RT}\bm{v}) \hspace{1cm} \forall \bm{v}\in S_{h}, \label{eq:SDG1} \\
B_h^*(\bm{u}_h, \bm{\psi})&=\nu^{-1}(\bm{\omega}_h, \bm{\psi}) \hspace{0.95cm} \forall \bm{\psi}\in V_h, \label{eq:SDG2} \\
b_h(\bm{u}_h, q)&=0 \hspace{2.5cm}\forall q\in P_h \label{eq:SDG3},
\end{align}
\end{subequations}
where the bilinear forms $B_h(\bm{\omega}_h,\bm{v})$ and $B_{h}^{*}(\bm{u}_h,\bm{\psi})$ are defined as
\begin{equation*}
\begin{split}
B_h(\bm{\omega}_h,\bm{v})&=-\sum_{e\in \mathcal{F}_p}( \bm{\omega}_h \bm{n},[\bm{v}])_e,\\
B_{h}^{*}(\bm{u}_h,\bm{\psi})&=\sum_{e\in \mathcal{F}_{u}^{0}} (\bm{u}_h,[\bm{\psi} \bm{n}])_e
\end{split}
\end{equation*}
and the bilinear forms $b_{h}^{*}(p_h,\bm{v})$ and $b_{h}(\bm{u}_h,q)$ are defined as
\begin{equation*}
\begin{split}
b_{h}^{*}(p_h,\bm{v})&=-\sum_{T\in \mathcal{T}_u}(p_h, \nabla\cdot \Pi^{RT}\bm{v})_T,\\
b_{h}(\bm{u}_h,q)&=-\sum_{e\in \mathcal{F}_{u}^{0}}(\bm{u}_h\cdot \bm{n},[q])_e.
\end{split}
\end{equation*}
Here, we list some important properties that will be employed later. First, integration by parts yields the following adjoint property
\begin{align}
B_h(\bm{\psi}, \bm{v})=B_h^*(\bm{v}, \bm{\psi}) \quad \forall (\bm{\psi}, \bm{v})\in V_h\times S_h. \label{adjoint}
\end{align}
Next, we notice that for $q\in P^0(T)$, $\forall T\in \mathcal{T}_u$, we have from integration by parts and the definition of $\Pi^{RT}$
\begin{align*}
( q, \nabla \cdot \Pi^{RT}\bm{v})_T &=  -(\nabla q,\Pi^{RT}\bm{v})_T+( q, \Pi^{RT}\bm{v}\cdot\bm{n})_{\partial T}\\
&=(q, \Pi^{RT}\bm{v}\cdot\bm{n})_{\partial T}=( q, \bm{v}\cdot\bm{n})_{\partial T},
\end{align*}
thereby summing up over all the elements $T\in \mathcal{T}_u$ yields the following adjoint property
\begin{align}
b_h^*(q,\bm{v})=b_h(\bm{v}, q) \quad \forall (q, \bm{v})\in P_h\times S_h.\label{adjoint-bh}
\end{align}
Finally,  the following inf-sup conditions hold (cf.
\cite{EricEngquistwave06,KimChung13}):
\begin{align}
\inf_{\bm{v}\in S_h}\sup_{\bm{\omega}\in V_h}\frac{B_h(\bm{\omega},\bm{v})}{\|\bm{v}\|_h\|\bm{\omega}\|_0}&\geq C,\label{eq:inf-sup-Bh}\\
\inf_{q\in M_h\backslash\{0\}}\sup_{\bm{v}\in
S_h\backslash\{0\}}\frac{b_h(\bm{v},
q)}{\|\bm{v}\|_h\|q\|_0}&\geq C.\label{eq:inf-sup-bh}
\end{align}

\begin{remark}(implementation).

Since the bilinear forms $b_h(\cdot,\cdot)$ and $b_h^*(\cdot,\cdot)$ are adjoint, we only need to compute $b_h(\cdot,\cdot)$ in the actual implementation. This bypasses the computation of the global velocity reconstruction in the bilinear form, thus our implementation only modifies the right hand side assembling compared to \cite{LinaParkShin19}.

\end{remark}

To facilitate later analysis, we define three interpolation operators $\mathcal{I}_h: [H^1(\Omega)]^2\rightarrow S_h$, $\pi_h: H^1(\Omega)\rightarrow M_h$ and $\mathcal{J}_h: [H^1(\Omega)]^{2\times 2}\rightarrow V_h$, which are given explicitly as
\begin{equation*}
\begin{split}
\mathcal{I}_h \bm{v}\mid_{D(e)}&=\frac{1}{|e|}\int_e\bm{v}\;ds\quad \forall e\in \mathcal{F}_{u},\\
\pi_h q\mid_{T}&=\frac{1}{|T|}\int_T q \;dx \quad \forall T\in \mathcal{T}_u,\\
\mathcal{J}_h \bm{\psi}\bm{n}_e\mid_e&=\frac{1}{|e|}\int_e \bm{\psi} \bm{n}_e\;ds\quad\forall  e\in \mathcal{F}_p.
\end{split}
\end{equation*}
It follows immediately that
\begin{align}
B_h^*(\mathcal{I}_h\bm{u}-\bm{u}, \bm{\psi})&=0 \quad \forall \bm{\psi}\in V_h,\label{errorIh}\\
B_h(\mathcal{J}_h \bm{\omega}-\bm{\omega}, \bm{v})&=0 \quad \forall \bm{v}\in S_h.\label{eq:Jhp}
\end{align}
In addition, standard interpolation error estimates (cf. \cite{Ciarlet78,EricEngquistwave06}) imply
\begin{equation}
\begin{split}
\|\mathcal{I}_h\bm{u}-\bm{u}\|_0&\leq C h\|\bm{u}\|_1\quad \bm{u}\in [H^1(\Omega)]^2,\\
\|\mathcal{J}_h\bm{\omega}-\bm{\omega}\|_0&\leq C h\|\bm{\omega}\|_1\quad \bm{\omega}\in [H^1(\Omega)]^{2\times 2},\\
\|p-\pi_hp\|_P&\leq C h\|p\|_1\quad p\in H^1(\Omega).
\end{split}
\label{eq:interpolation-error}
\end{equation}


Next we will present some properties which play an important role for later analysis. To this end, we decompose $\bm{f}$ as (cf. \cite{Arnold18})
\begin{align}
\bm{f}=\bm{g}+\lambda \nabla \chi\label{eq:decompose-f},
\end{align}
where $\bm{g}$ is the curl of a function in $H(\text{curl};\Omega)$ whose tangent trace vanishes on $\partial \Omega$. $\chi\in H^1(\Omega)$ is such that $\|\nabla \chi\|_0=1$ and $\lambda\in \mathbb{R}^+$.

\begin{lemma}
For $\bm{v}\in S_h$ and $p\in H^1(\Omega)$, it holds
\begin{align}
(\nabla p, \Pi^{RT}\bm{v}) =b_h^*(\pi_h p, \bm{v}).\label{eq:p-property}
\end{align}
Thus, we have velocity invariance property, i.e.,
\begin{align}
(\bm{f}, \Pi^{RT}\bm{v})=(\bm{g},\Pi^{RT}\bm{v})+b_h^*(\pi_h \chi, \lambda \bm{v}).\label{eq:velocityinvariance}
\end{align}
In addition, $\Pi^{RT}\bm{u}_h$ is divergence free.

\end{lemma}

\begin{proof}
For $\bm{v}\in S_h$, we have from integration by parts
\begin{equation}
\begin{split}
(\nabla p, \Pi^{RT}\bm{v}) & = \sum_{T\in \mathcal{T}_u}\Big( -(p,\nabla\cdot \Pi^{RT}\bm{v})_T+( p, \Pi^{RT}\bm{v}\cdot\bm{n})_{\partial T}\Big)\\
&= -\sum_{T\in \mathcal{T}_u} (\pi_h p, \nabla\cdot \Pi^{RT}\bm{v})_T\\
&=b_h^*(\pi_h p, \bm{v}),
\end{split}
\label{eq:property1}
\end{equation}
which gives \eqref{eq:p-property}.
Thereby, we can obtain
\begin{align*}
(\nabla \chi, \Pi^{RT}\bm{v})=b_h^*(\pi_h\chi, \bm{v})\quad \forall \bm{v}\in S_h,
\end{align*}
which together with \eqref{eq:decompose-f} yields \eqref{eq:velocityinvariance}.

Finally, we will show that $\Pi^{RT}\bm{u}_h$ is divergence free. For $T\in \mathcal{T}_u$, we have from integration by parts and the definition of $\Pi^{RT}$
\begin{align*}
(\nabla\cdot \Pi^{RT}\bm{u}_h,q)_T=\sum_{e\in \partial T}( \Pi^{RT}\bm{u}_h\cdot\bm{n},q)_{e}=\sum_{e\in \partial T}( \bm{u}_h\cdot\bm{n},q)_{e}\quad \forall q\in P^0(T).
\end{align*}
In addition, we can infer from \eqref{eq:SDG3} that
\begin{align*}
-\sum_{e\in \partial T}(\bm{u}_h\cdot\bm{n},1)_e=0.
\end{align*}
This gives
\begin{align*}
(\nabla\cdot \Pi^{RT}\bm{u}_h,q)_T=0\quad \forall q\in P^0(T),
\end{align*}
which implies $\nabla\cdot \Pi^{RT}\bm{u}_h=0$ in $T$. On the other hand, $\bm{u}_h$ is continuous over $e\in \mathcal{F}_u$, thus $\Pi^{RT}\bm{u}_h\in H(\text{div};\Omega)$, which implies that $\Pi^{RT}\bm{u}_h$ is divergence free.

\end{proof}

\begin{remark}(comparison to existing methods).

We compare our method with weak Galerkin (WG) method proposed in \cite{Mu20}.
First of all, our method shares the same degrees of freedom for pressure with WG proposed in \cite{Mu20}, see Figures~\ref{dof} and \ref{dof:WG}. Our method has less degrees of freedom for velocity since our velocity space only consists of edge degrees of freedom while WG consists of edge and interior degrees of freedom (cf. Figure~\ref{dof:WG}). On the other hand, our approach is based on the first order system, thus velocity gradient can be calculated simultaneously. However, WG is based on the primal formulation, straightforward calculation of velocity gradient is not available.


\begin{figure}[t]
\centering
\includegraphics[width =0.4\textwidth]{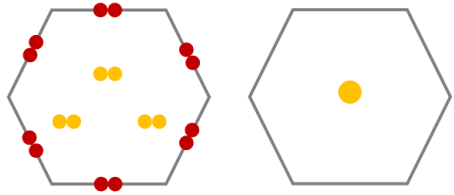}
\caption{Schematic of degrees of freedom for velocity and pressure for WG method over the polygon.}
\label{dof:WG}
\end{figure}




\end{remark}

\begin{lemma}\label{lemma:approximation}
(approximation properties).
We have for $\bm{v}\in S_h$
\begin{align*}
\|\bm{v}-\Pi^{RT}\bm{v}\|_0\leq C h\|\bm{v}\|_h.
\end{align*}

\end{lemma}

\begin{proof}

For any $T\in \mathcal{T}_u$, notice that $\bm{v}$ is a constant over the sub-triangle of $T$ and we use $\bm{v}_i$ to denote the value of $\bm{v}$ restricted to each sub-triangle $\tau\subset T$. Let $\bar{\bm{v}}=\frac{1}{|T|}\sum_{i=1}^m\int_\tau\bm{v}_i$ and we use $\|\bm{v}\|_{h,T}$ to represent $\|\bm{v}\|_h$ restricted to $T$, then we have
\begin{align}
\|\bm{v}-\bar{\bm{v}}\|_{0,T}\leq C h_T\|\bm{v}\|_{h,T},\label{eqbarv}
\end{align}
where we use the fact that if the right side vanishes, then the left side also vanishes. An application of scaling arguments yields the desired estimate.

Since $\Pi^{RT} \bar{\bm{v}}=\bar{\bm{v}}$ for $\bar{\bm{v}}\in [P^0(T)]^2$, we can proceed analogously to \eqref{eqbarv} to obtain
%
\begin{align*}
\|\Pi^{RT}\bar{\bm{v}}-\Pi^{RT}\bm{v}\|_{0,T}=\| \sum_{i=1}^m(\bar{\bm{v}}-\bm{v}_i)\bm{\varphi}_i\|_{0,T}\leq C \sum_{i=1}^m |\bar{\bm{v}}-\bm{v}_i|\|\bm{\varphi}_i\|_{0,T}\leq C h_T\|\bm{v}\|_{h,T},
\end{align*}
where Lemma~\ref{lemma:Hdiv} is employed in the last inequality.

Finally, the triangle inequality yields
\begin{align*}
\|\bm{v}-\Pi^{RT}\bm{v}\|_{0,T}\leq \|\bm{v}-\bar{\bm{v}}\|_{0,T}+\|\Pi^{RT}\bar{\bm{v}}-\Pi^{RT}\bm{v}\|_{0,T}.
\end{align*}
Summing up over all the primal elements $T\in \mathcal{T}_u$ yields the desired estimate.

\end{proof}

\section{A priori error analysis}\label{sec:analysis}

In this section we will present the convergence estimates for all the variables involved. In particular we will prove that the error estimates for velocity are independent of $\nu$. The main contribution of this section is to prove the superconvergence of velocity via duality argument, which is non-trivial.

The unique solvability of the discrete formulation \eqref{eq:SDG} can be proved in line with \cite{LinaParkShin19}. The proof will not be repeated here for the sake of conciseness. We will give the uniform a priori bound for velocity in the next lemma.

\begin{lemma}(uniform a priori bound on the discrete velocity).
The following estimate holds
\begin{align}
\|\bm{u}_h\|_h\leq C \nu^{-1}\|\bm{g}\|_0.\label{eq:apirori}
\end{align}

\end{lemma}

\begin{proof}

Invoking \eqref{eq:decompose-f} and \eqref{eq:velocityinvariance}, we can rewrite \eqref{eq:SDG} as
\begin{align}
B_h(\bm{\omega}_h,\bm{v})+b_h^*(p_h,\bm{v})&=(\bm{g}, \Pi^{RT}\bm{v})+\lambda b_h^*(\pi_h\chi,\bm{v}) \quad \forall \bm{v}\in S_{h}, \label{eq:SDG11} \\
B_h^*(\bm{u}_h, \bm{\psi})&=\nu^{-1}(\bm{\omega}_h, \bm{\psi}) \hspace{2.4cm} \forall \bm{\psi}\in V_h, \label{eq:SDG12} \\
b_h(\bm{u}_h, q)&=0 \hspace{4cm}\forall q\in M_h \label{eq:SDG13}.
\end{align}
Taking $\bm{v}=\bm{u}_h, \bm{\psi}=\bm{\omega}_h,q=p_h-\lambda \pi_h\chi$ in \eqref{eq:SDG11}, \eqref{eq:SDG12} and \eqref{eq:SDG13}, and sum up, we can obtain
\begin{align}
\nu^{-1}\|\bm{\omega}_h\|_0^2=(\bm{g},\Pi^{RT}\bm{u}_h).\label{eq:omega}
\end{align}
We can infer from \eqref{adjoint}, \eqref{eq:inf-sup-Bh} and \eqref{eq:SDG12} that
\begin{align*}
\|\bm{u}_h\|_h\leq C \sup_{\bm{\psi}\in V_h}\frac{B_h(\bm{\psi},\bm{u}_h)}{\|\bm{\psi}\|_0}=C\sup_{\bm{\psi}\in V_h}\frac{B_h^*(\bm{u}_h,\bm{\psi})}{\|\bm{\psi}\|_0}=C\sup_{\bm{\psi}\in V_h}\frac{\nu^{-1}(\bm{\omega}_h,\bm{\psi})}{\|\bm{\psi}\|_0}\leq C \nu^{-1}\|\bm{\omega}_h\|_0,
\end{align*}
thereby, we can infer from Lemma~\ref{lemma:approximation} and \eqref{eq:omega}
\begin{align*}
\|\bm{u}_h\|_h^2\leq C \nu^{-2}\|\bm{\omega}_h\|_0^2\leq C \nu^{-1} \|\bm{g}\|_0\|\Pi^{RT}\bm{u}_h\|_0\leq C \nu^{-1} \|\bm{g}\|_0\|\bm{u}_h\|_h,
\end{align*}
which yields the desired estimate by dividing both sides by $\|\bm{u}_h\|_h$.

\end{proof}

\begin{remark}
Contrary to the classical estimate $\|\bm{u}_h\|_h\leq C \nu^{-1}\|\bm{f}\|_0$ that can be obtained from \cite{KimChung13} and  \cite{LinaParkShin19}, the a priori bound \eqref{eq:apirori} persists in the limit $\lambda \rightarrow \infty$. This bound can be incorporated into Navier-Stokes equations to establish error estimates under smallness assumption that only concerns the solenoidal part $\bm{g}$ of the body force. This can improve the existing bound given in \cite{ChungQiu17}, where smallness assumption for $\bm{f}$ is required.

\end{remark}
%

\begin{theorem}\label{thm:energy}

Let $(\bm{\omega},\bm{u},p)\in [H^1(\Omega)]^{2\times 2}\times[H^1(\Omega)]^2\times H^1(\Omega)$ be the weak solution of \eqref{firstorder}, $\Delta \bm{u}\in [L^2(\Omega)]^{2}$ and let $(\bm{\omega}_h,\bm{u}_h,p_h)\in V_h\times S_h\times P_h$ be the numerical solution of \eqref{eq:SDG}. Then the following estimates hold
\begin{align*}
\|\bm{w}-\bm{w}_h\|_0&\leq C h(\nu \|\Delta \bm{u}\|_0+\|\bm{\omega}\|_1),\\
\|\mathcal{I}_h\bm{u}-\bm{u}_h\|_h &\leq C  h(\|\bm{u}\|_2+ \|\Delta\bm{u}\|_{0}),\\
\|p-p_h\|_0&\leq C h (\|p\|_1+\nu \|\Delta \bm{u}\|_0+\|\bm{\omega}\|_1).
\end{align*}

\end{theorem}

\begin{proof}


%
First, we define $\tilde{\bm{\omega}}\in V_h$ such that
\begin{align}
\nu^{-1}(\tilde{\bm{\omega}}, \bm{\psi})=B_h^*(\mathcal{I}_h \bm{u}, \bm{\psi}) \quad\forall \bm{\psi}\in V_h.\label{eq:dpt}
\end{align}
Invoking \eqref{eq:SDG2} and \eqref{eq:dpt}, we can obtain
\begin{align}
\nu^{-1}(\tilde{\bm{\omega}}-\bm{\omega}_h, \bm{\psi})=
B_h^*(\mathcal{I}_h \bm{u}-\bm{u}_h, \bm{\psi}) \quad \forall \bm{\psi}\in V_h.\label{eq:erroromega}
\end{align}
We can infer from \eqref{errorIh} and \eqref{eq:dpt} that
\begin{align*}
\nu^{-1}(\tilde{\bm{\omega}},\bm{\psi})=B_h^*(\mathcal{I}_h \bm{u},\bm{\psi})=B_h^*(\bm{u}, \bm{\psi})=\nu^{-1}(\bm{\omega}, \bm{\psi})\quad\forall\bm{\psi}\in V_h,
\end{align*}
which means $\bm{\tilde{\omega}}$ is the $L^2$-orthogonal projection of $\bm{\omega}$ onto $V_h$.
Thus,
\begin{align*}
\|\tilde{\bm{\omega}}-\bm{\omega}\|_0\leq \|\bm{\omega}-\mathcal{J}_h \bm{\omega}\|_0.
\end{align*}
An appeal to \eqref{adjoint}, \eqref{eq:inf-sup-Bh} and \eqref{eq:erroromega} yields
\begin{align}
\|\mathcal{I}_h \bm{u}-\bm{u}_h\|_h\leq C\sup_{\bm{\psi}\in V_h}\frac{B_h(\bm{\psi},\mathcal{I}_h \bm{u}-\bm{u}_h)}{\|\bm{\psi}\|_0}=C\sup_{\bm{\psi}\in V_h}\frac{B_h^*(\mathcal{I}_h \bm{u}-\bm{u}_h,\bm{\psi})}{\|\bm{\psi}\|_0}\leq C \nu^{-1} \|\tilde{\bm{\omega}}-\bm{\omega}_h\|_0.\label{eq:compare}
\end{align}
On the other hand, we can infer from the first equation of \eqref{mixed} and \eqref{eq:property1} that
\begin{align*}
(\bm{f},\Pi^{RT}\bm{v})& = (-\nu\Delta \bm{u}+\nabla p, \Pi^{RT}\bm{v}) = -\nu(\Delta \bm{u}, \Pi^{RT}\bm{v}) + b_h^*(\pi_h p, \bm{v})\quad \forall \bm{v}\in S_h.
\end{align*}
Integration by parts yields
\begin{align*}
B_h(\bm{w},\bm{v}) = -(\bm{v}, \nabla\cdot\bm{w})=-\nu(\bm{v}, \Delta \bm{u})\quad \forall \bm{v}\in S_h.
\end{align*}
Thereby, we can obtain
\begin{align}
B_h(\bm{w}, \bm{v})+b_h^*(\pi_h p, \bm{v})=(\bm{f},\Pi^{RT}\bm{v})+\nu(\Delta \bm{u}, \Pi^{RT}\bm{v}-\bm{v})\quad \forall \bm{v}\in S_h.\label{eq:interp1}
\end{align}
On the other hand, we have from the discrete formulation \eqref{eq:SDG1}
\begin{align*}
B_h(\bm{w}_h, \bm{v}) + b_h^*(p_h, \bm{v})=(\bm{f},\Pi^{RT}\bm{v})\quad \forall \bm{v}\in S_h,
\end{align*}
which can be combined with \eqref{eq:interp1} yielding
\begin{align*}
B_h(\bm{w}-\bm{w}_h, \bm{v})+b_h^*(\pi_h p-p_h,\bm{v})=(\nu\Delta \bm{u}, \Pi^{RT}\bm{v}-\bm{v})\quad \forall \bm{v}\in S_h.
\end{align*}
Hence
\begin{align}
B_h(\mathcal{J}_h\bm{w}-\bm{w}_h, \bm{v})+b_h^*(\pi_h p-p_h,\bm{v})=\nu(\Delta \bm{u}, \Pi^{RT}\bm{v}-\bm{v})\quad \forall \bm{v}\in S_h\label{eq:error1}.
\end{align}
Taking $\bm{v} = \mathcal{I}_h\bm{u}-\bm{u}_h$ in \eqref{eq:error1}, then it leads to
\begin{align*}
B_h(\mathcal{J}_h\bm{w}-\bm{w}_h, \mathcal{I}_h\bm{u}-\bm{u}_h)+b_h^*(\pi_h p-p_h, \mathcal{I}_h\bm{u}-\bm{u}_h)=\nu(\Delta \bm{u}, \Pi^{RT}(\mathcal{I}_h\bm{u}-\bm{u}_h)-(\mathcal{I}_h\bm{u}-\bm{u}_h)).
\end{align*}
Since $ b_h^*(\pi_h p-p_h, \mathcal{I}_h\bm{u}-\bm{u}_h)=b_h(\mathcal{I}_h\bm{u}-\bm{u}_h,\pi_h p-p_h)=0$, we have
\begin{align}
B_h(\bm{w}-\bm{w}_h, \mathcal{I}_h\bm{u}-\bm{u}_h)=\nu(\Delta \bm{u}, \Pi^{RT}(\mathcal{I}_h\bm{u}-\bm{u}_h)-(\mathcal{I}_h\bm{u}-\bm{u}_h))\label{consistency}.
\end{align}
It follows from \eqref{adjoint}, \eqref{eq:Jhp}, \eqref{eq:erroromega} and \eqref{consistency} that
\begin{equation*}
\begin{split}
\nu^{-1}\|\tilde{\bm{\omega}}-\bm{\omega}_h\|_0^2&
=B_h(\tilde{\bm{\omega}}-\bm{\omega}_h, \mathcal{I}_h \bm{u}-\bm{u}_h)\\
&=B_h(\tilde{\bm{\omega}}-\bm{\omega}, \mathcal{I}_h\bm{u}-\bm{u}_h)+\nu(\Delta \bm{u}, \Pi^{RT}(\mathcal{I}_h\bm{u}-\bm{u}_h)-(\mathcal{I}_h\bm{u}-\bm{u}_h))\\
&=B_h(\tilde{\bm{\omega}}-\mathcal{J}_h\bm{\omega}, \mathcal{I}_h\bm{u}-\bm{u}_h)+\nu(\Delta \bm{u} , \Pi^{RT}(\mathcal{I}_h\bm{u}-\bm{u}_h)-(\mathcal{I}_h\bm{u}-\bm{u}_h)),
\end{split}
\end{equation*}
therefore, we have from Lemma~\ref{lemma:approximation} and \eqref{eq:compare}
\begin{align*}
\|\tilde{\bm{\omega}}-\bm{\omega}_h\|_0^2&\leq C \Big( \nu\|\tilde{\bm{\omega}}-\mathcal{J}_h\bm{\omega}\|_0\|\mathcal{I}_h\bm{u}-\bm{u}_h\|_h
+\nu^2 h\|\Delta\bm{u}\|_{0}\|\mathcal{I}_h\bm{u}-\bm{u}_h\|_h\Big)\\
&\leq C \Big( \|\tilde{\bm{\omega}}-\mathcal{J}_h\bm{\omega}\|_0\|\tilde{\bm{\omega}}-\bm{\omega}_h\|_0
+ \nu h\|\Delta\bm{u}\|_{0}\|\tilde{\bm{\omega}}-\bm{\omega}_h\|_0\Big).
\end{align*}
The triangle inequality yields
\begin{align*}
\|\tilde{\bm{\omega}}-\mathcal{J}_h\bm{\omega}\|_0\leq \|\tilde{\bm{\omega}}-\bm{\omega}\|_0+\|\bm{\omega}-\mathcal{J}_h\bm{\omega}\|_0\leq 2 \|\bm{\omega}-\mathcal{J}_h\bm{\omega}\|_0.
\end{align*}
Thus, we can conclude that
\begin{align*}
\|\tilde{\bm{\omega}}-\bm{\omega}_h\|_0\leq C h(\|\bm{\omega}\|_1+\nu \|\Delta\bm{u}\|_{0}),
\end{align*}
which gives
\begin{align*}
\|\bm{\omega}-\bm{\omega}_h\|_0\leq  C h(\|\bm{\omega}\|_1+\nu \|\Delta\bm{u}\|_{0}).
\end{align*}
Invoking \eqref{eq:compare}, we can obtain
\begin{align}
\|\mathcal{I}_h\bm{u}-\bm{u}_h\|_h\leq C\nu^{-1}\|\tilde{\bm{\omega}}-\bm{\omega}_h\|_0\leq  C h(\nu^{-1}\|\bm{\omega}\|_1+ \|\Delta\bm{u}\|_{0})\leq C h(\|\bm{u}\|_2+ \|\Delta\bm{u}\|_{0})\label{eq:uz}.
\end{align}
Next, we consider the error estimate for $\|p-p_h\|_0$. The discrete adjoint property \eqref{adjoint-bh} and the inf-sup condition \eqref{eq:inf-sup-bh} imply
\begin{align} \label{estimate_ph_php}
\|p_h-\pi_h p\|_0\leq C \sup_{\bm{v}\in S_h\setminus \{0\}}\frac{b_h(\bm{v},p_h-\pi_h p)}{\|\bm{v}\|_h}=C \sup_{\bm{v}\in S_h\setminus \{0\}}\frac{b_h^*(p_h-\pi_h p,\bm{v})}{\|\bm{v}\|_h}.
\end{align}
Moreover, \eqref{eq:error1} yields
\begin{equation}
\begin{split}
b_h^*(p_h-\pi_h p,\bm{v})&=B_h(\mathcal{J}_h\bm{\omega}-\bm{\omega}_h, \bm{v})-\nu(\Delta \bm{u}, \Pi^{RT}\bm{v}-\bm{v})\quad \forall \bm{v}\in S_h.
\end{split}
\label{bhbyBh}
\end{equation}
We have from \eqref{eq:scaling-psi}, Lemma~\ref{lemma:approximation}, \eqref{estimate_ph_php} and \eqref{bhbyBh} that
\begin{equation*}
\begin{split}
\|p_h-\pi_h p\|_0&\leq C(\|\mathcal{J}_h \bm{\omega}-\bm{\omega}_h\|_{0}+\nu h\|\Delta \bm{u}\|_{0}).
\end{split}
\end{equation*}
Therefore, we can obtain
\begin{align*}
\|p-p_h\|_0&\leq C(\|p-\pi_h p\|_0+\|\bm{\omega}_h-\mathcal{J}_h \bm{\omega}\|_0+\nu h\|\Delta \bm{u}\|_0)\\
&\leq C h(\|p\|_1+\|\bm{\omega}\|_1+\nu \|\Delta\bm{u}\|_{0}).
\end{align*}

\end{proof}

The $L^2$ error estimate for velocity can be stated as follows.
\begin{theorem}
Let $(\bm{\omega},\bm{u})\in [H^1(\Omega)]^{2\times 2}\times [H^1(\Omega)]^2$ be the weak solution of \eqref{firstorder} and $\Delta \bm{u}\in [L^2(\Omega)]^2$, and let $\bm{u}_h\in S_h$ be the numerical solution of \eqref{eq:SDG}. Then the following estimate holds
\begin{align*}
\|\bm{u}-\bm{u}_h\|_0\leq Ch \Big( \|\Delta \bm{u}\|_0+\|\bm{u}\|_2\Big).
\end{align*}

\end{theorem}

\begin{proof}

We have from the discrete Poincar\'{e} inequality (cf. \cite{Brenner03}) and \eqref{eq:uz} that
\begin{align*}
\|\mathcal{I}_h\bm{u}-\bm{u}_h\|_0\leq C \|\mathcal{I}_h\bm{u}-\bm{u}_h\|_h\leq C  h(\|\bm{u}\|_2+ \|\Delta\bm{u}\|_{0}).
\end{align*}
Then we can infer from the triangle inequality and \eqref{eq:interpolation-error}
\begin{align*}
\|\bm{u}-\bm{u}_h\|_0\leq \|\mathcal{I}_h\bm{u}-\bm{u}\|_0+\|\mathcal{I}_h\bm{u}-\bm{u}_h\|_0\leq C h\Big(\|\bm{u}\|_2+\|\Delta\bm{u}\|_{0}\Big).
\end{align*}

\end{proof}


Now we state the superconvergence for velocity.
\begin{theorem}(superconvergence).\label{thm:super}
Let $(\bm{\omega},\bm{u})\in [H^1(\Omega)]^{2\times 2}\times [H^1(\Omega)]^2$ be the weak solution of \eqref{firstorder}, $\Delta \bm{u}\in [H^2(\Omega)]^{2}$ and let $\bm{u}_h\in S_h$ be the numerical solution of \eqref{eq:SDG}. Then the following estimate holds
\begin{align*}
\|\mathcal{I}_h\bm{u}-\bm{u}_h\|_0\leq C h^2(\|\bm{u}\|_2+\|\Delta \bm{u}\|_2).
\end{align*}

\end{theorem}

\begin{proof}

Given a right hand side $\bm{r}\in [L^2(\Omega)]^2$, let $(\bm{\omega}_r,\bm{u}_r,p_r)$ denote the solution of
\begin{align}
\nabla \cdot \bm{\omega}_r+\nabla p_r&=\bm{r} \hspace{1.2cm} \mbox{in}\;\Omega\label{eq:dual1},\\
\bm{\omega}_r&=-\nu \nabla \bm{u}_r  \quad \mbox{in}\;\Omega,\label{eq:dual2}\\
\nabla\cdot \bm{u}_r & =0 \hspace{1.3cm}\mbox{in}\;\Omega,\\
\bm{u}_r &=\bm{0}\hspace{1.3cm}\mbox{on}\;\partial \Omega\label{eq:dual14}
\end{align}
and let $(\bm{\omega}_{r,h},\bm{u}_{r,h},p_{r,h})\in V_h\times S_h\times P_h$ denote the solution of
\begin{align}
-B_h(\bm{\omega}_{r,h},\bm{v})+b_h^*(p_{r,h},\bm{v})&=(\bm{r}, \Pi^{RT}\bm{v})\qquad \forall \bm{v}\in S_h\label{eq:discretedual1},\\
-B_h^*(\bm{u}_{r,h}, \bm{\psi})&=\nu^{-1}(\bm{\omega}_{r,h},\bm{\psi})\quad \forall \bm{\psi}\in V_h,\\
b_h(\bm{u}_{r,h},q)=&0\hspace{2.5cm} \forall q\in P_h.\label{eq:dual-discrete3}
\end{align}
We obtain from classical regularity results for the incompressible Stokes equations (cf. \cite{Girault86}) that
\begin{equation}
\begin{split}
\nu \|\bm{u}_r\|_2&\leq C \|\bm{r}\|_0,\\
\|\bm{\omega}_r\|_1&\leq C \|\bm{r}\|_0.
\end{split}
\label{eq:regularity}
\end{equation}
We can now apply a duality argument to prove the superconvergence. We have
\begin{align*}
\|\mathcal{I}_h\bm{u}-\bm{u}_h\|_0=\sup_{\bm{r}\in [L^2(\Omega)]^2\backslash\{\bm{0}\}}\frac{(\bm{r},\mathcal{I}_h\bm{u}-\bm{u}_h)}{\|\bm{r}\|_0}
\end{align*}
%
and
\begin{align}
(\bm{r},\mathcal{I}_h\bm{u}-\bm{u}_h)&=(\bm{r}, \mathcal{I}_h\bm{u}-\bm{u}_h-\Pi^{RT}(\mathcal{I}_h\bm{u}-\bm{u}_h))
+(\bm{r},\Pi^{RT}(\mathcal{I}_h\bm{u}-\bm{u}_h)).\label{eq:rdecomposition}
\end{align}
The first term on the right hand side can be estimated by Lemma~\ref{lemma:approximation} and \eqref{eq:uz}
\begin{align*}
(\bm{r}, \mathcal{I}_h\bm{u}-\bm{u}_h-\Pi^{RT}(\mathcal{I}_h\bm{u}-\bm{u}_h))\leq Ch \|\bm{r}\|_0\|\mathcal{I}_h\bm{u}-\bm{u}_h\|_h\leq C h^2\|\bm{r}\|_0( \|\Delta \bm{u}\|_0+\|\bm{u}\|_2).
\end{align*}
It remains to estimate the second term of \eqref{eq:rdecomposition}. We have from \eqref{eq:dual2}, \eqref{eq:discretedual1}, integration by parts and the fact that $b_h^*(p_{r,h},\mathcal{I}_h\bm{u}-\bm{u}_h)=0$
\begin{equation}
\begin{split}
(\bm{r}, \Pi^{RT}(\mathcal{I}_h\bm{u}-\bm{u}_h))& = -B_h(\bm{\omega}_{r,h}, \mathcal{I}_h\bm{u}-\bm{u}_h)+\nu^{-1}(\bm{\omega}_r, \bm{\omega}-\bm{\omega}_h)+(\nabla \bm{u}_r, \bm{\omega}-\bm{\omega}_h)\\
&=-B_h(\bm{\omega}_{r,h}, \mathcal{I}_h\bm{u}-\bm{u}_h)+\nu^{-1}(\bm{\omega}_r, \bm{\omega}-\bm{\omega}_h)+(\nabla (\bm{u}_r-\mathcal{I}_h\bm{u}_r), \bm{\omega}-\bm{\omega}_h)\\
&=-B_h(\bm{\omega}_{r,h}, \mathcal{I}_h\bm{u}-\bm{u}_h)+\nu^{-1}(\bm{\omega}_r, \bm{\omega}-\bm{\omega}_h)+\sum_{e\in \mathcal{F}_p}((\bm{\omega}-\bm{\omega}_h)\bm{n}, [\bm{u}_r-\mathcal{I}_h\bm{u}_r])_e\\
&\;+\sum_{e\in \mathcal{F}_u}([(\bm{\omega}-\bm{\omega}_h)\bm{n}], \bm{u}_r-\mathcal{I}_h\bm{u}_r)_e
-(\bm{u}_r-\mathcal{I}_h\bm{u}_r, \nabla\cdot \bm{\omega})\\
&=-B_h(\bm{\omega}_{r,h}, \mathcal{I}_h\bm{u}-\bm{u}_h)+\nu^{-1}(\bm{\omega}_r, \bm{\omega}-\bm{\omega}_h)\\
&\;+B_h(\bm{\omega}-\bm{\omega}_h, \mathcal{I}_h\bm{u}_r)-(\bm{u}_r-\mathcal{I}_h\bm{u}_r, \nabla\cdot \bm{\omega}).
\end{split}
\label{eq:rnorm}
\end{equation}
It follows from \eqref{eq:SDG2} and \eqref{eq:error1} by taking $\bm{\psi}=\mathcal{J}_h\bm{\omega}_r$ and $\bm{v}=\mathcal{I}_h\bm{u}_r$, respectively
\begin{align*}
B_h^*(\mathcal{I}_h\bm{u}-\bm{u}_h, \mathcal{J}_h\bm{\omega}_r)&=\nu^{-1}(\bm{\omega}-\bm{\omega}_h, \mathcal{J}_h\bm{\omega}_r),\\
B_h(\bm{\omega}-\bm{\omega}_h, \mathcal{I}_h\bm{u}_r) &=\nu (\Delta \bm{u} ,\Pi^{RT}\mathcal{I}_h\bm{u}_r-\mathcal{I}_h\bm{u}_r),
\end{align*}
where we use $b_h^*(\pi_h p-p_h, \mathcal{I}_h\bm{u}_r)=0$ in the second equality.

Therefore, we can recast \eqref{eq:rnorm} into the following form
\begin{align*}
(\bm{r}, \Pi^{RT}(\mathcal{I}_h\bm{u}-\bm{u}_h))&=-B_h(\bm{\omega}_{r,h}-\mathcal{J}_h\bm{\omega}_r, \mathcal{I}_h\bm{u}-\bm{u}_h)+\nu^{-1}(\bm{\omega}_r-\mathcal{J}_h\bm{\omega}_r, \bm{\omega}-\bm{\omega}_h)\\
&\;+\nu(\Delta \bm{u}, \Pi^{RT}\mathcal{I}_h\bm{u}_r-\mathcal{I}_h\bm{u}_r)-(\bm{u}_r-\mathcal{I}_h\bm{u}_r, \nabla\cdot \bm{\omega})\\
&=-B_h(\bm{\omega}_{r,h}-\mathcal{J}_h\bm{\omega}_r, \mathcal{I}_h\bm{u}-\bm{u}_h)+\nu^{-1}(\bm{\omega}_r-\mathcal{J}_h\bm{\omega}_r, \bm{\omega}-\bm{\omega}_h)\\
&\;+\nu(\Delta \bm{u}, \Pi^{RT}\mathcal{I}_h\bm{u}_r-\mathcal{I}_h\bm{u}_r)-\nu(\Delta \bm{u}, \bm{u}_r-\mathcal{I}_h\bm{u}_r)\\
&=-B_h(\bm{\omega}_{r,h}-\mathcal{J}_h\bm{\omega}_r, \mathcal{I}_h\bm{u}-\bm{u}_h)+\nu^{-1}(\bm{\omega}_r-\mathcal{J}_h\bm{\omega}_r, \bm{\omega}-\bm{\omega}_h)+\nu(\Delta \bm{u}, \Pi^{RT}\mathcal{I}_h\bm{u}_r-\bm{u}_r)\\
&:=\sum_{i=1}^3 I_i.
\end{align*}
First we can obtain the following estimate by proceeding analogously to Theorem~\ref{thm:energy} for the dual problem \eqref{eq:dual1}-\eqref{eq:dual-discrete3}
\begin{align}
\|\bm{\omega}_{r,h}-\mathcal{J}_h\bm{\omega}_r\|_0\leq C\Big(\|\bm{\omega}_r-\mathcal{J}_h\bm{\omega}_r\|_0+\nu h\|\Delta \bm{u}_r\|_0\Big)\leq C h\Big(\|\bm{\omega}_r\|_1+\nu\|\Delta \bm{u}_r\|_0\Big)\leq C h\|\bm{r}\|_0.\label{eq:estimate-dual1}
\end{align}
Then $I_1$ can be estimated by the Cauchy-Schwarz inequality, \eqref{eq:uz} and \eqref{eq:estimate-dual1}
\begin{align*}
|I_1|\leq C \|\bm{\omega}_{r,h}-\mathcal{J}_h\bm{\omega}_r\|_0\|\mathcal{I}_h\bm{u}-\bm{u}_h\|_h\leq h^2\|\bm{r}\|_0(\|\Delta \bm{u}\|_0+\|\bm{u}\|_2).
\end{align*}
We can bound $I_2$ by the Cauchy-Schwarz inequality and Theorem~\ref{thm:energy}
\begin{align*}
|I_2|\leq \nu^{-1}\|\bm{\omega}_r-\mathcal{J}_h\bm{\omega}_r\|_0\|\bm{\omega}-\bm{\omega}_h\|_0\leq C\nu^{-1} h^2\|\bm{w}_r\|_1(\nu \|\Delta \bm{u}\|_0+\|\bm{\omega}\|_1)\leq Ch^2\|\bm{r}\|_0(\|\Delta \bm{u}\|_0+\nu^{-1}\|\bm{\omega}\|_1).
\end{align*}
For an arbitrary $T\in \mathcal{T}_u$ with $m$ edges, we have from the definitions of $\Pi^{RT}$ and $\mathcal{I}_h$
\begin{align*}
\Pi^{RT}\mathcal{I}_h\bm{u}_r = \sum_{i=1}^m (\frac{1}{|e|}\int_e \mathcal{I}_h\bm{u}_r\cdot\bm{n}\;ds)\bm{\varphi}_i=\sum_{i=1}^m (\frac{1}{|e|}\int_e \bm{u}_r\cdot\bm{n}\;ds)\bm{\varphi}_i.
\end{align*}
Thus, we can conclude that $\Pi^{RT}\mathcal{I}_h\bm{u}_r=\Pi^{RT}\bm{u}_r$. Thereby we can rewrite $I_3$ as
\begin{align*}
I_3&=\nu(\Delta \bm{u}, \Pi^{RT}\mathcal{I}_h\bm{u}_r-\bm{u}_r)=\nu(\Delta \bm{u}, \Pi^{RT}\bm{u}_r-\bm{u}_r)\\
&=\nu(\Delta \bm{u}-\pi_h\Delta \bm{u}, \Pi^{RT}\bm{u}_r-\bm{u}_r)+\nu(\pi_h\Delta \bm{u}, \Pi^{RT}\bm{u}_r-\bm{u}_r).
\end{align*}
For the first summand, Lemma~\ref{lemma:Hdiv} and \eqref{eq:regularity} give the desired bound
\begin{align*}
\nu(\Delta \bm{u}-\pi_h\Delta \bm{u}, \Pi^{RT}\bm{u}_r-\bm{u}_r)&\leq \nu \|\Delta \bm{u}-\pi_h\Delta \bm{u}\|_0\|\Pi^{RT}\bm{u}_r-\bm{u}_r\|_0\\
&\leq C\nu h^2 \|\Delta \bm{u}\|_1\|\bm{u}_r\|_1\leq C  h^2 \|\Delta \bm{u}\|_1\|\bm{r}\|_0.
\end{align*}

The bound for the second summand is inspired by the work given in \cite{LinkeMerdon17}. First, we notice that $\nabla \cdot \bm{u}_r=0$, hence there exists a function $\sigma_r$ such that $\text{curl}\,\sigma_r=\bm{u}_r$. Further, since $\bm{u}_r\in [H^1(\Omega)]^2$, it holds $\nu\|\sigma_r\|_2\leq C \|\bm{r}\|_0$, see, e.g., \cite{Hiptmair02}. Furthermore, there exists a space $\widetilde{S}_h$ consisting of piecewise linear polynomials and recall that $\Pi^0$ is the nodal value interpolation operator based on the generalized barycentric coordinates (cf. Lemma~\ref{lemma:Hdiv}). Since $\Pi^{RT}$ satisfies \eqref{eq:Ih}, it holds $\bm{u}_r-\Pi^{RT}\bm{u}_r=\text{curl}(\sigma_r-\Pi^0\sigma_r)$
\begin{align*}
\nu(\pi_h \Delta\bm{u},\bm{u}_r-\Pi^{RT}\bm{u}_r)&=\nu\sum_{T\in \mathcal{T}_u}(\pi_h \Delta\bm{u},\text{curl}(\sigma_r-\Pi^0\sigma_r) )_T\\
&\leq |\sum_{T\in \mathcal{T}_u}\nu(\text{curl}(\pi_h\Delta \bm{u}), \sigma_r-\Pi^0\sigma_r)_T|+\nu|\sum_{T\in \mathcal{T}_u}(\pi_h \Delta\bm{u}\cdot\bm{t}, \sigma_r-\Pi^0\sigma_r)_{\partial T}|
\end{align*}
The first term on the right hand side vanishes since $\pi_h \Delta \bm{u}$ is piecewise constant. We can bound the second term by the error estimates on convex polygon based on the generalized barycentric coordinates (cf. \cite{Gillette12})
\begin{align*}
|(\pi_h \Delta\bm{u}\cdot\bm{t}, \sigma_r-\Pi^0\sigma_r)_{\partial T}|&\leq \|\pi_h \Delta\bm{u}\|_{L^\infty(\partial T)}\|\sigma_r-\Pi^0\sigma_r\|_{L^1(\partial T)}\leq C h_T^2\|\Delta\bm{u}\|_{L^\infty(\partial T)}\|\sigma_r\|_{2,T}.
\end{align*}
Thus
\begin{align*}
\nu(\pi_h \Delta\bm{u},\bm{u}_r-\Pi^{RT}\bm{u}_r)\leq C \nu h^2\sum_{T\in \mathcal{T}_u}\|\Delta \bm{u}\|_{2,T}\|\sigma_r\|_{2,T}\leq C h^2\|\Delta\bm{u}\|_2\|\bm{r}\|_0.
\end{align*}
Combining the preceding arguments, we can achieve the desired estimate.

\end{proof}

\begin{remark}
The derivation of superconvergence for lowest order SDG method is non-trivial. In general we need to invoke some nonstandard trace inequalities in order to deliver the desirable results, see \cite{LinaParkShin19}. Here, we are able to achieve the superconvergence without resorting to nonstandard trace inequality and the key idea lies in the use of the divergence preserving operator in the discrete formulation of the dual problem as well as the integration of continuous formulation and discrete formulation in \eqref{eq:rnorm}. We emphasize that the sole application of continuous formulation (cf. \eqref{eq:dual1}-\eqref{eq:dual14}) or discrete formulation (cf. \eqref{eq:discretedual1}-\eqref{eq:dual-discrete3}) can not deliver desirable result. Importantly, our analysis will provide new perspectives for the proof of superconvergence for other discretizations as well.

\end{remark}

\begin{remark}\label{remark:con}
We can obtain the following estimate by using the formulation given in \cite{LinaParkShin19}
\begin{align*}
\|\bm{\omega}-\bm{\omega}_h\|_0&\leq C h(\nu\|\bm{u}\|_2+\|p\|_1),\\
\|\mathcal{I}_h\bm{u}-\bm{u}_h\|_h&\leq Ch( \|\bm{u}\|_2+\frac{1}{\nu}\|p\|_1),\\
\|p-p_h\|_0&\leq C  h(\nu\|\bm{u}\|_2+\|p\|_1).
\end{align*}
Then the discrete Poincar\'{e} inequality (cf. \cite{Brenner03}) yields
\begin{align*}
\|\mathcal{I}_h\bm{u}-\bm{u}_h\|_0\leq C \|\mathcal{I}_h\bm{u}-\bm{u}_h\|_h\leq Ch( \|\bm{u}\|_2+\frac{1}{\nu}\|p\|_1).
\end{align*}
We can observe from the above estimates that the velocity error will grow unboundedly as $\nu\rightarrow 0$, and therefore the formulation from \cite{LinaParkShin19} generates unsatisfactory solution as $\nu\rightarrow 0$.

\end{remark}

\section{Numerical experiments}\label{sec:numerical}

In this section several numerical tests will be tested to confirm the proposed theories. In the following tests, we will employ three types of meshes: Unstructured triangular mesh, trapezoidal mesh and polygonal mesh shown in Figure~\ref{mesh}. The accuracy and robustness of the proposed method will be investigated. For the sake of simplicity we denote the present formulation (cf. \eqref{eq:SDG}) as SDG1 and the formulation obtained from \cite{LinaParkShin19} as SDG2.

\begin{figure}[t]
    \centering
    \includegraphics[width=0.32\textwidth]{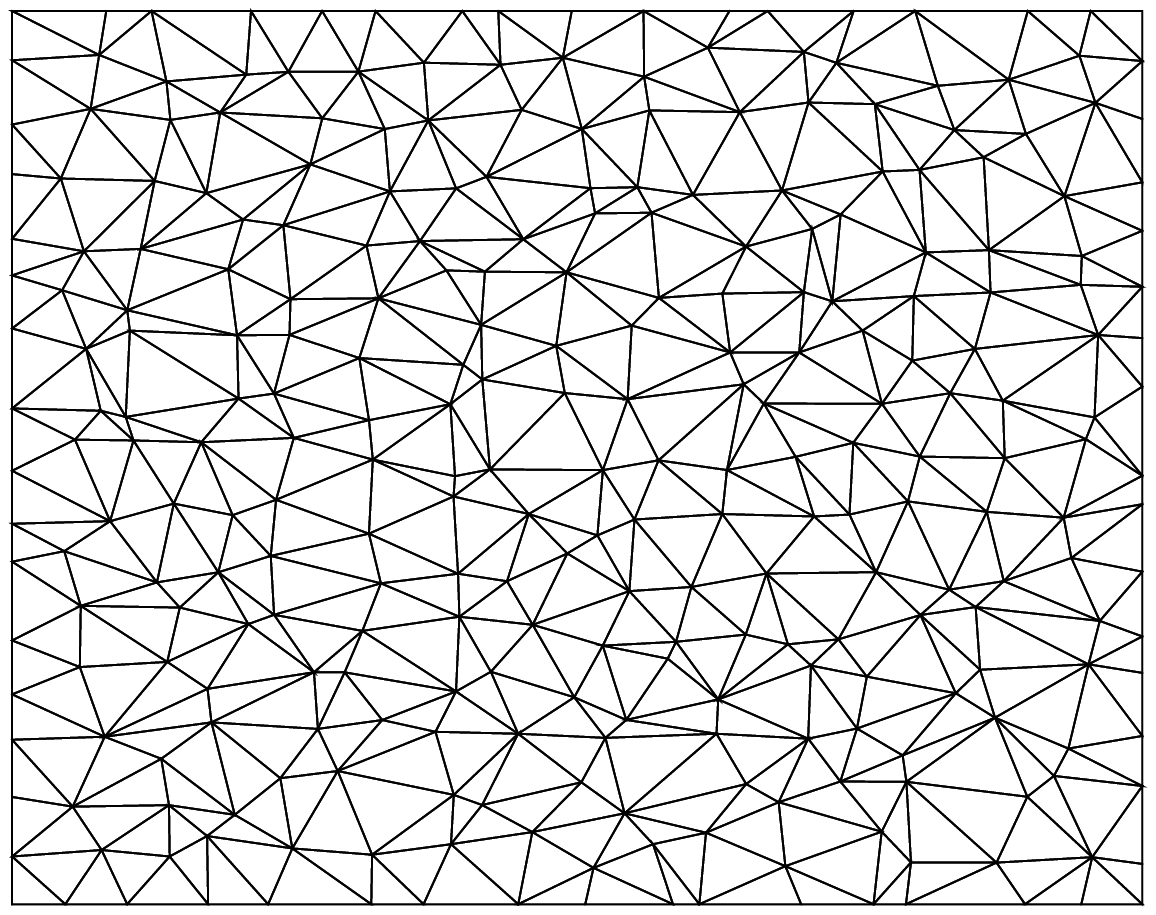}
    \includegraphics[width=0.32\textwidth]{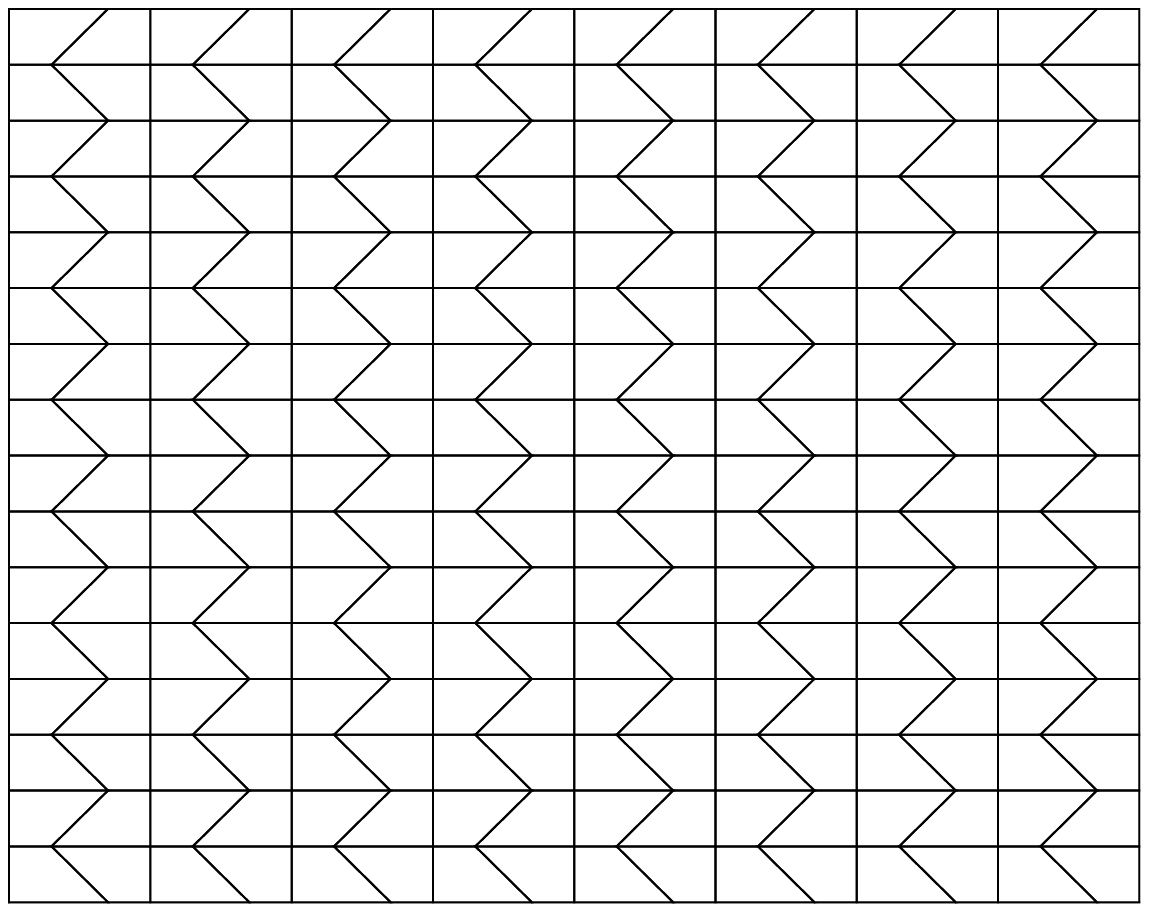}
     \includegraphics[width=0.32\textwidth]{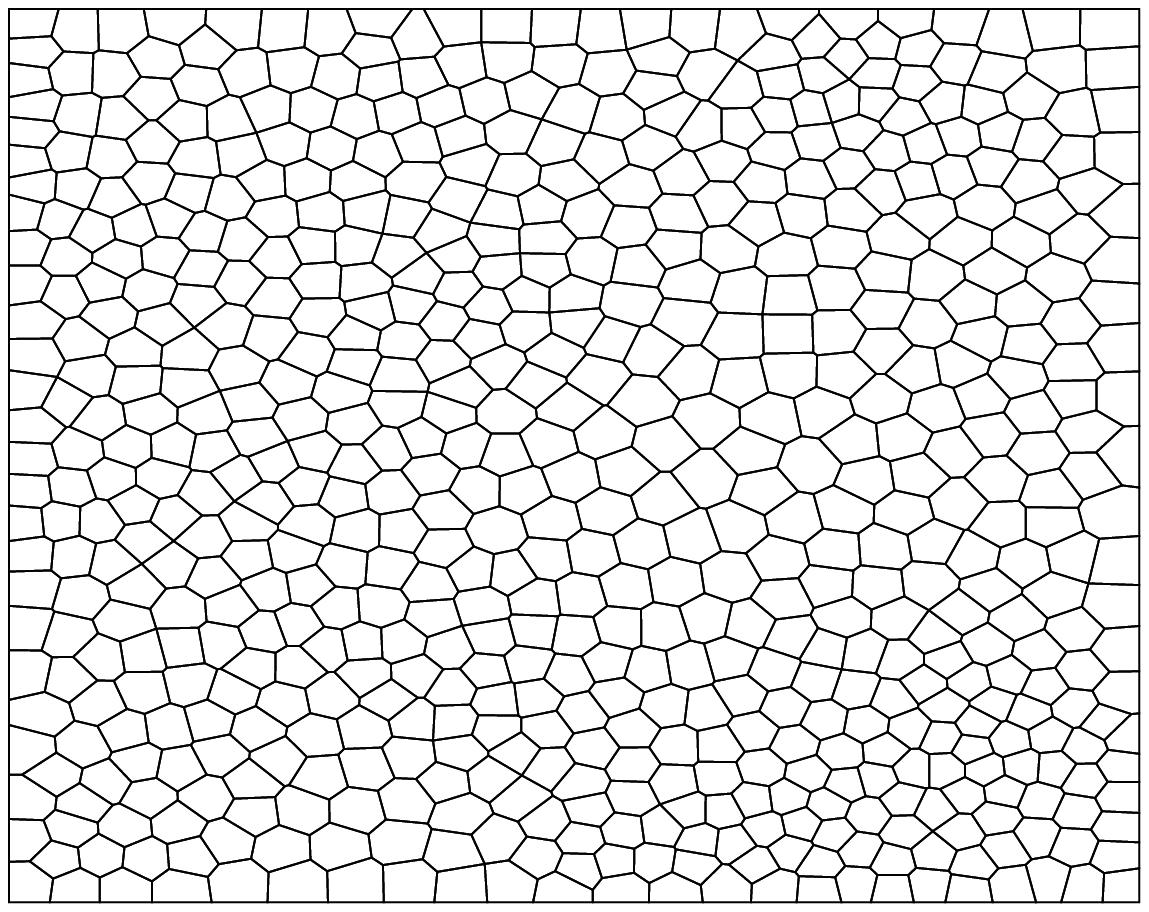}
    \caption{Three types of meshes used in numerical experiments: Unstructured triangular mesh (left), trapezoidal mesh (middle), and polygonal mesh (right).}
    \label{mesh}
\end{figure}

\subsection{Unstructured triangular mesh}\label{ex:tri}

\subsubsection{Accuracy and robustness test}\label{ex1:first}

Let $\Omega=(0,1)^2$ and let the exact solution be given by
\begin{align}
\bm{u}=\left(
         \begin{array}{c}
           \pi x^2(1-x)^2\sin(2\pi y) + 1 \\
           -2x(1-x)(1-2x)\sin(\pi y)^2 + 1 \\
         \end{array}
       \right),\quad p= \sin(x)\cos(y)+(\cos(1)-1) \sin(1).\label{eq:exact1}
\end{align}
We use unstructured triangular meshes in this section for the numerical simulation, see Figure~\ref{mesh}. The convergence history against the mesh size for $\nu=1$ are plotted in Figure~\ref{ex1:con}. We can observe that first order convergence can be achieved in $L^2$ errors of $\bm{u},p$ and $\bm{\omega}$ for both algorithms; in addition, second order convergence can be obtained for $\|\mathcal{I}_h\bm{u}-\bm{u}_h\|_0$, which confirms the theoretical results presented in Theorem~\ref{thm:energy} and Theorem~\ref{thm:super}.

Then we perform simulations to test the robustness of our method. To this end, we fix $h=1/16$ and choose $\nu=10^2,10,1,10^{-1},10^{-2},10^{-3},10^{-4},10^{-5},10^{-6}$. The right hand side can be calculated by $\bm{f}=-\nu \Delta \bm{u}+\nabla p$. $L^2$ errors for velocity, pressure and velocity gradient are reported in Figure~\ref{ex1:accuracy}. The velocity error deteriorates for $\nu\rightarrow 0$ and is asymptotically proportional to $1/\nu$ (when $\nu \leq 1$) as predicted by the theory of SDG2 (cf. Remark~\ref{remark:con}), which indicates that SDG2 is not pressure robust. On the contrary, the velocity error of SDG1 remains a constant for various values of $\nu$, which validates the independence of velocity on the pressure variable. In addition, we can observe similar performances for pressure error from SDG1 and SDG2. Moreover, the $L^2$ error of velocity gradient from SDG1 is asymptotically proportional to $\nu$, whereas, the $L^2$ error of velocity gradient from SDG2 tends to be a constant when $\nu\leq 1$, which is consistent with the theories given in Theorem~\ref{thm:energy} and Remark~\ref{remark:con}.

Finally, we display the numerical approximations for $\nu=10^{-6}$ for both algorithms in Figure~\ref{ex1:numerical-sol}. It is easy to see that SDG1 yields correct numerical approximation for velocity while SDG2 yields wrong numerical approximation for velocity. Numerical approximation for pressure is correct for both algorithms. To further verify the robustness of SDG1, we show the convergence history for $\nu=10^{-6}$, and we can observe first order convergence in $L^2$ errors of velocity, pressure and velocity gradient; in addition, superconvergence can be obtained for $\|\mathcal{I}_h\bm{u}-\bm{u}_h\|_0$.

\begin{figure}[t]
    \centering
    \includegraphics[width=0.35\textwidth]{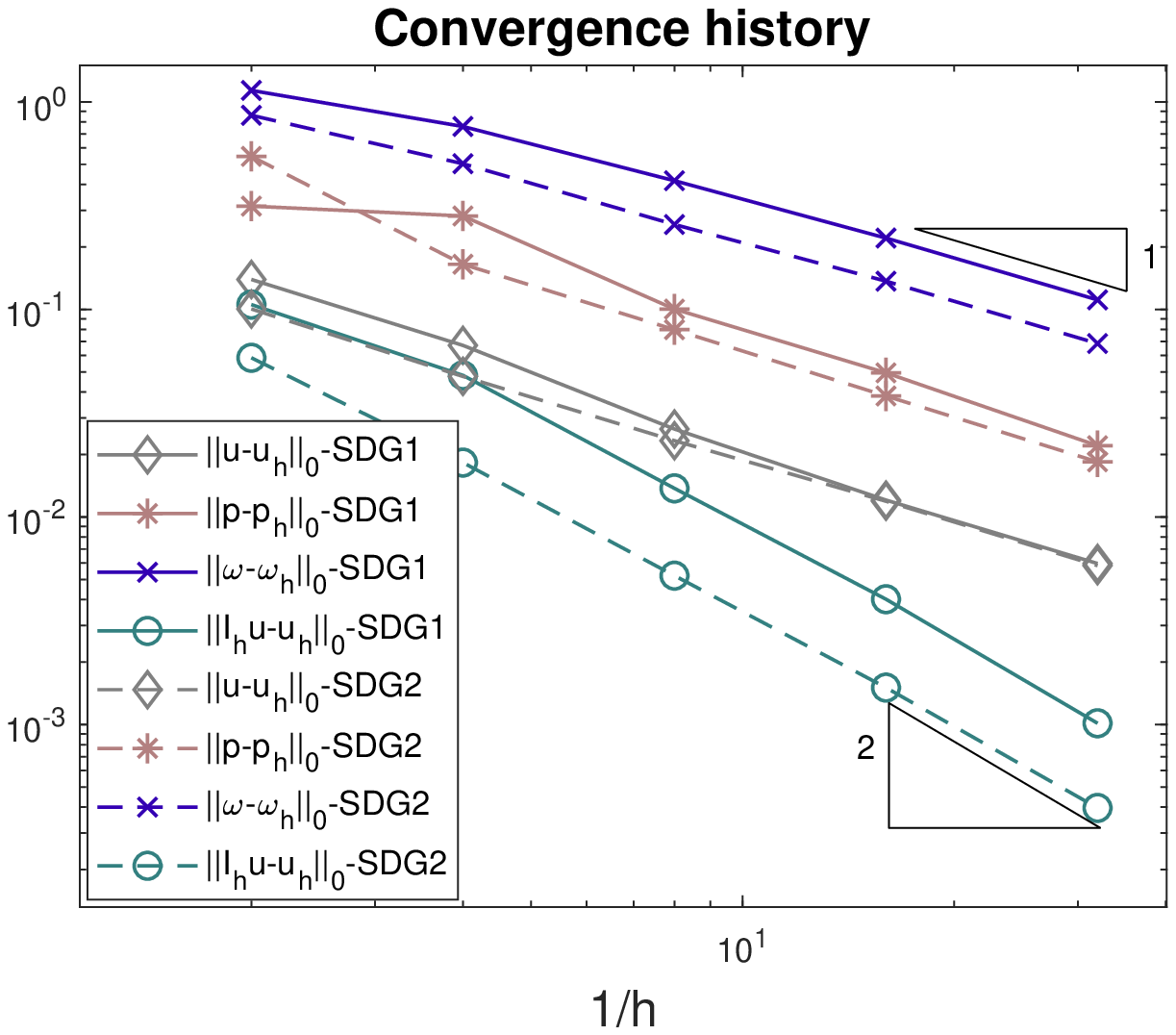}
     \includegraphics[width=0.35\textwidth]{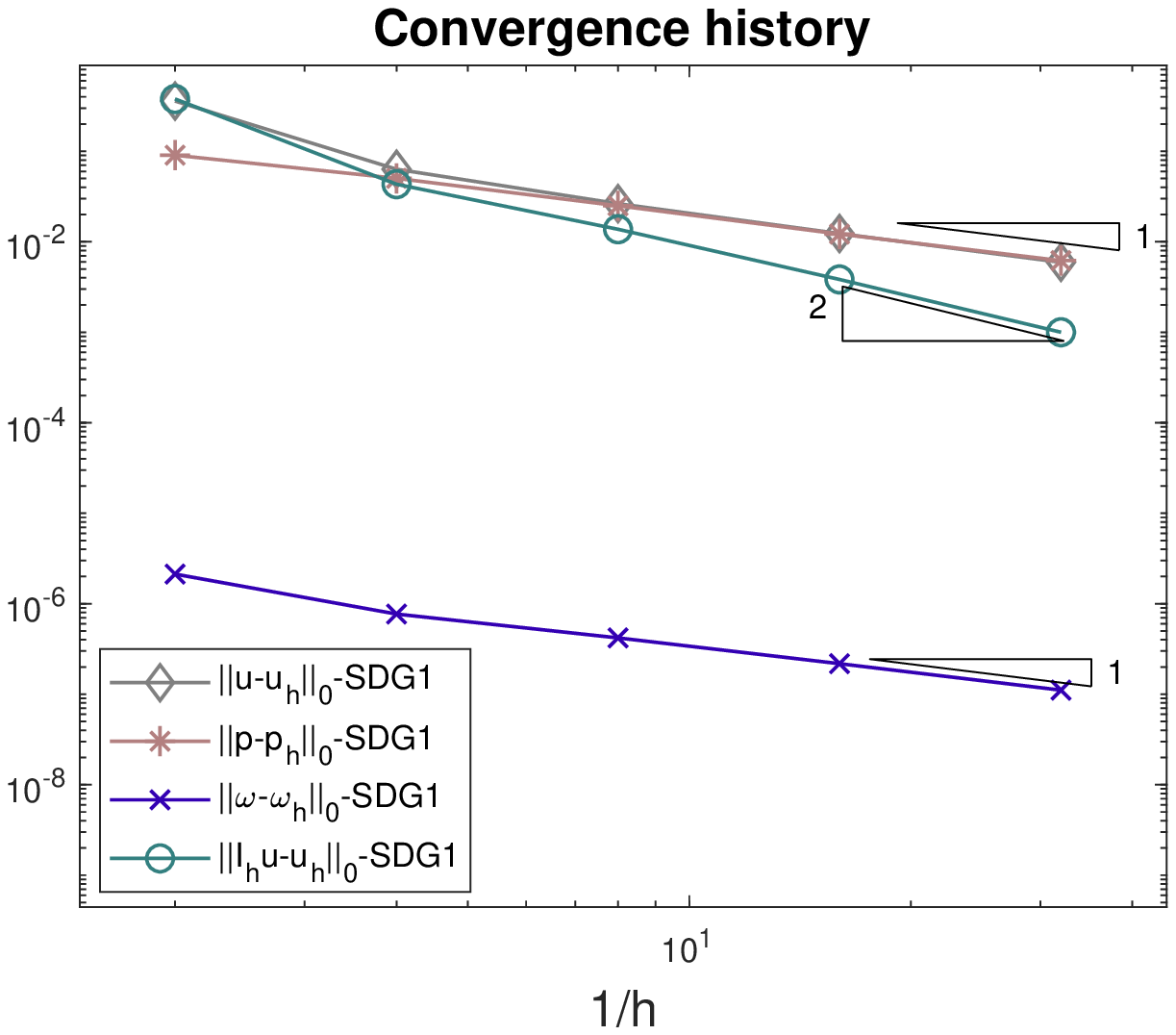}
    \caption{Example~\ref{ex1:first}: Convergence history for $\nu=1$ (left) and $\nu=10^{-6}$ (right).}
    \label{ex1:con}
\end{figure}

\begin{figure}[t]
    \centering
    \includegraphics[width=0.32\textwidth]{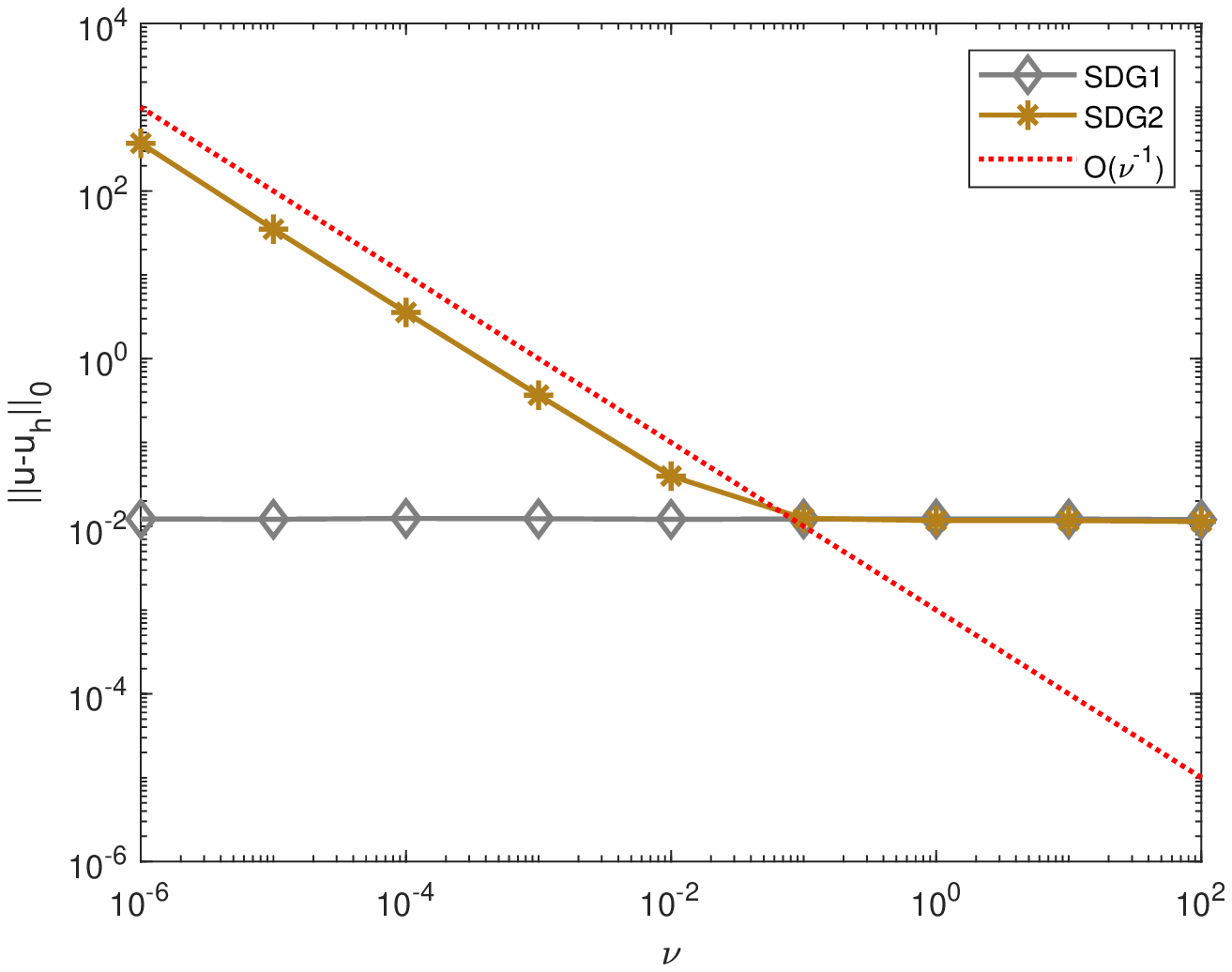}
    \includegraphics[width=0.32\textwidth]{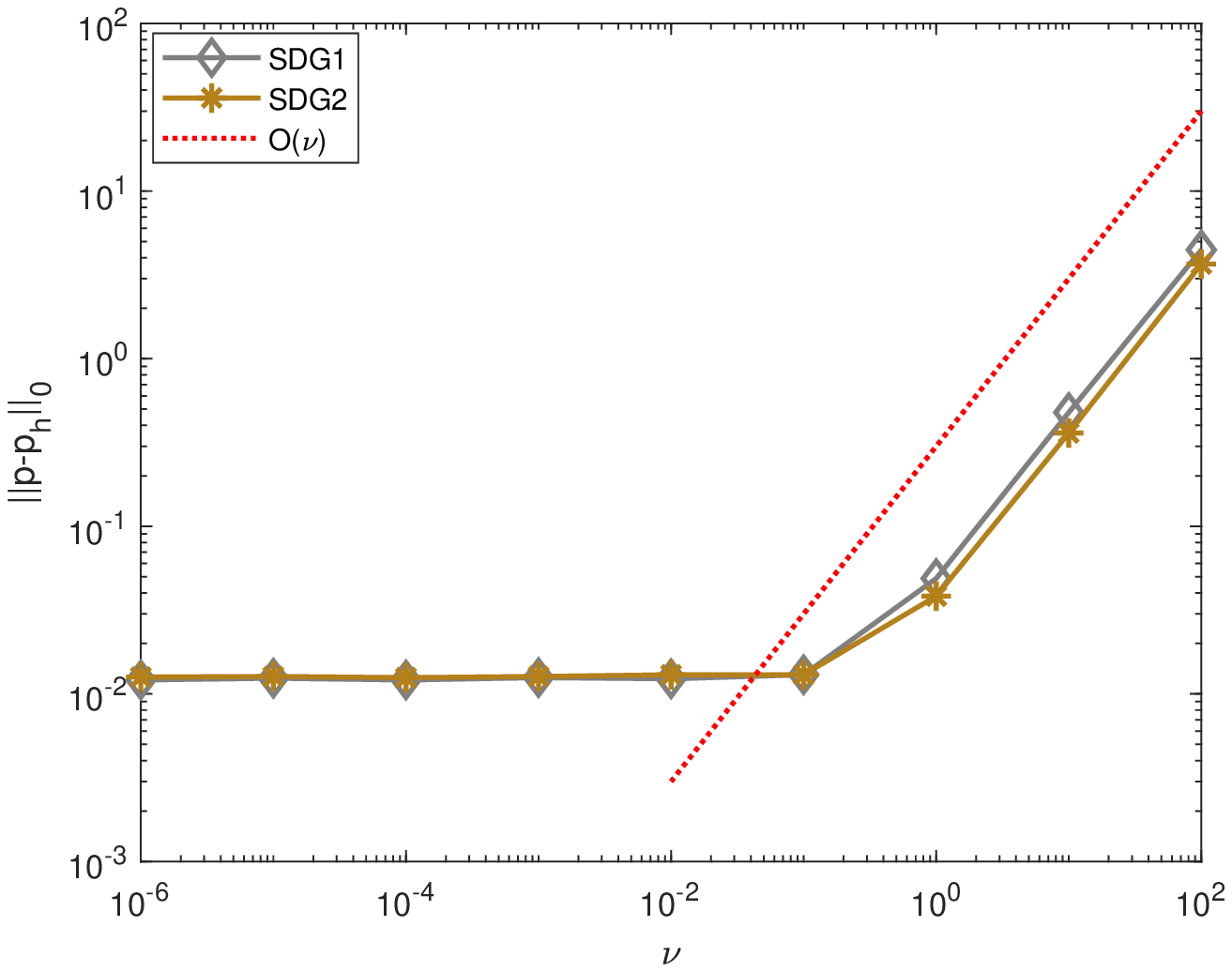}
     \includegraphics[width=0.32\textwidth]{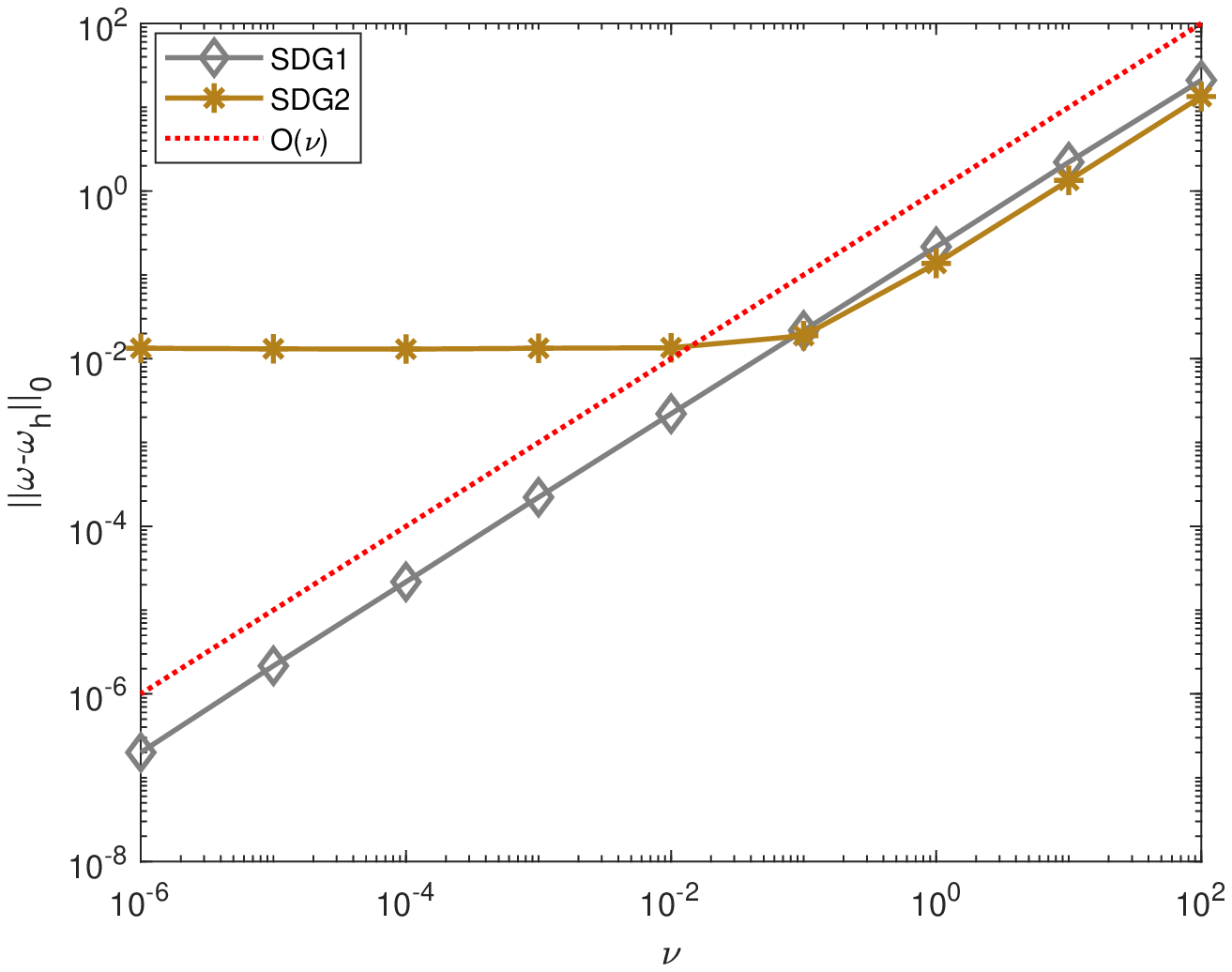}
    \caption{Example~\ref{ex1:first}: Error profiles for velocity (left), pressure (middle) and velocity gradient (right) on unstructured triangular mesh with $h=1/16$.}
    \label{ex1:accuracy}
\end{figure}

\begin{figure}[t]
    \centering
    \includegraphics[width=0.32\textwidth]{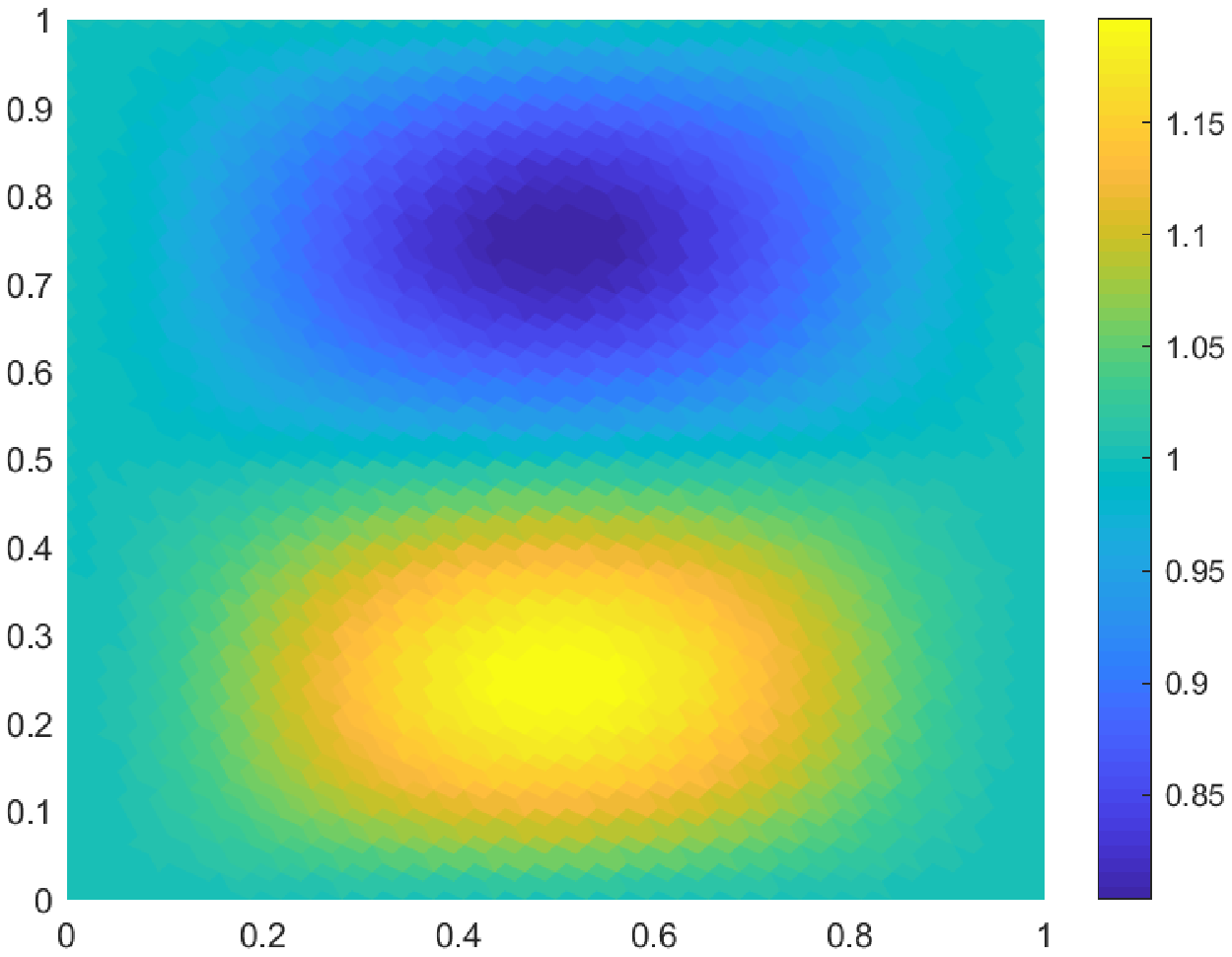}
    \includegraphics[width=0.32\textwidth]{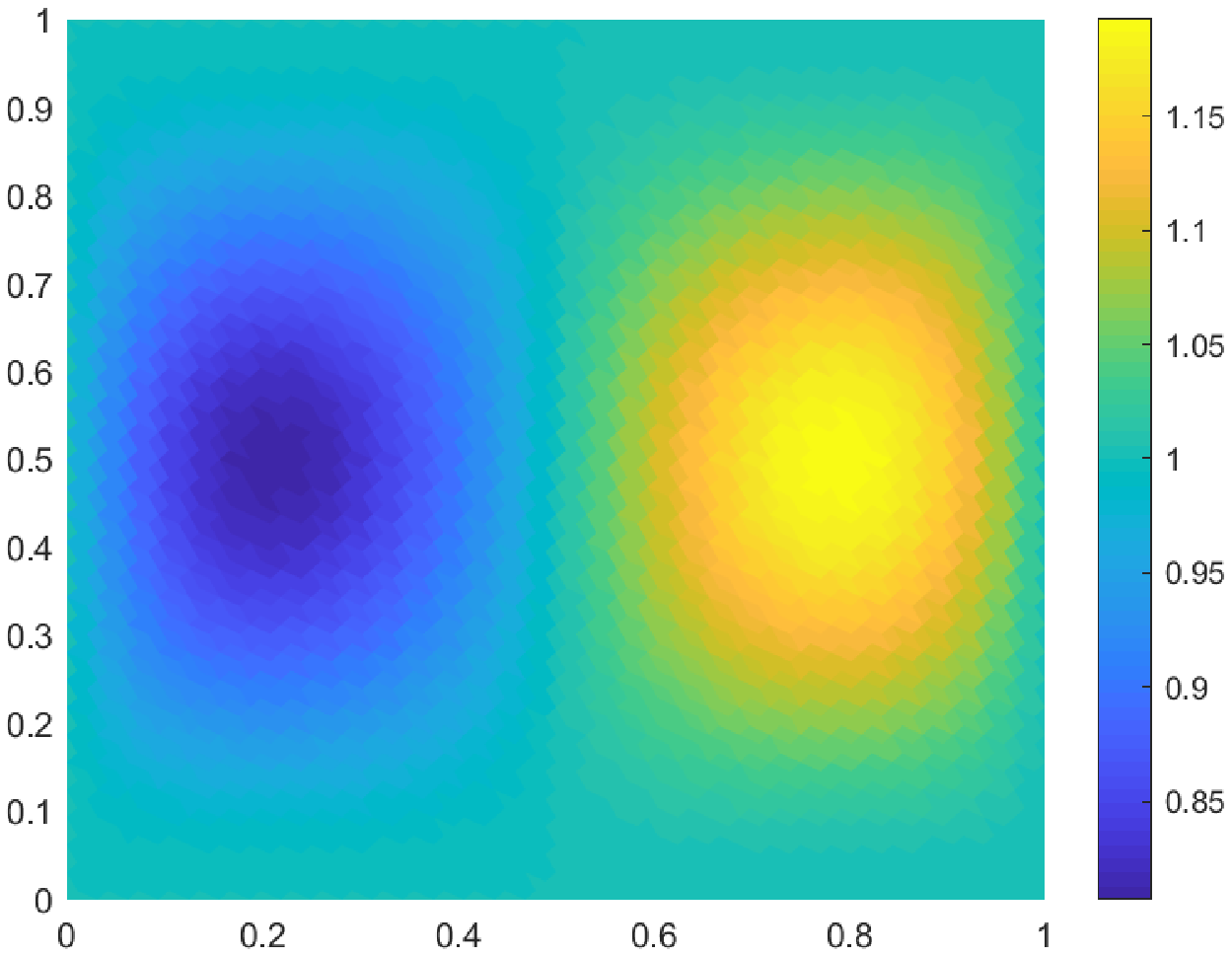}
    \includegraphics[width=0.32\textwidth]{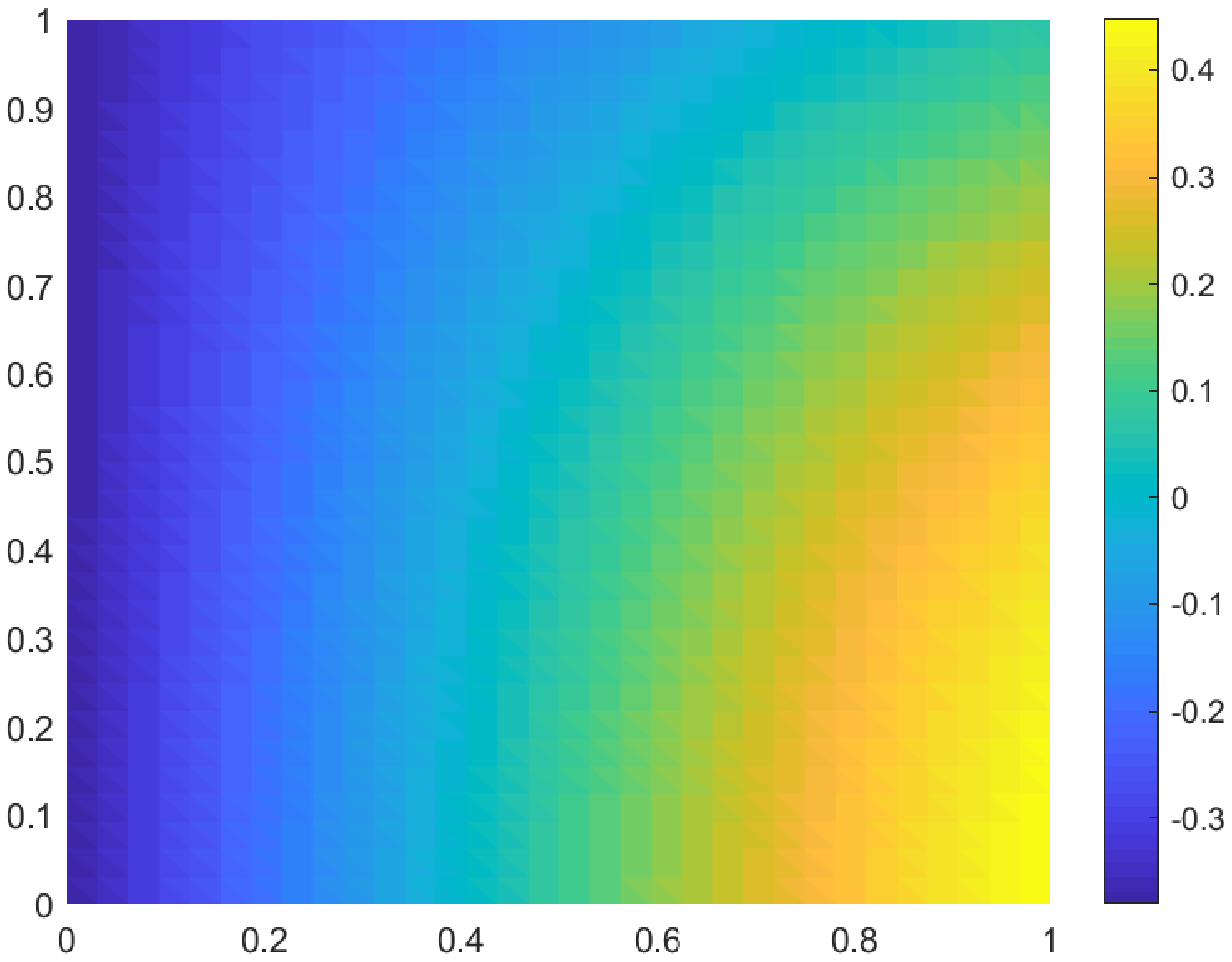}\\
    \includegraphics[width=0.32\textwidth]{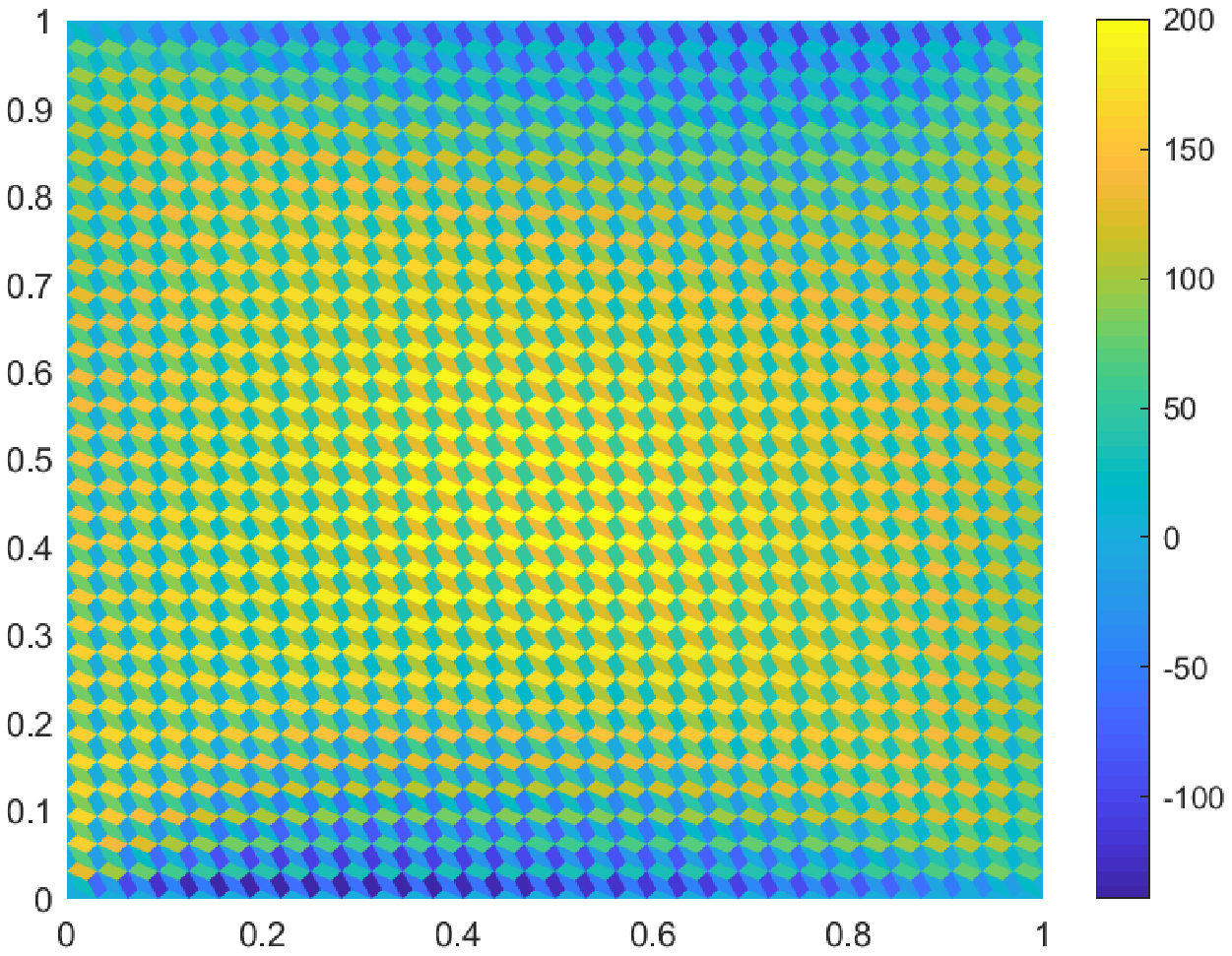}
    \includegraphics[width=0.32\textwidth]{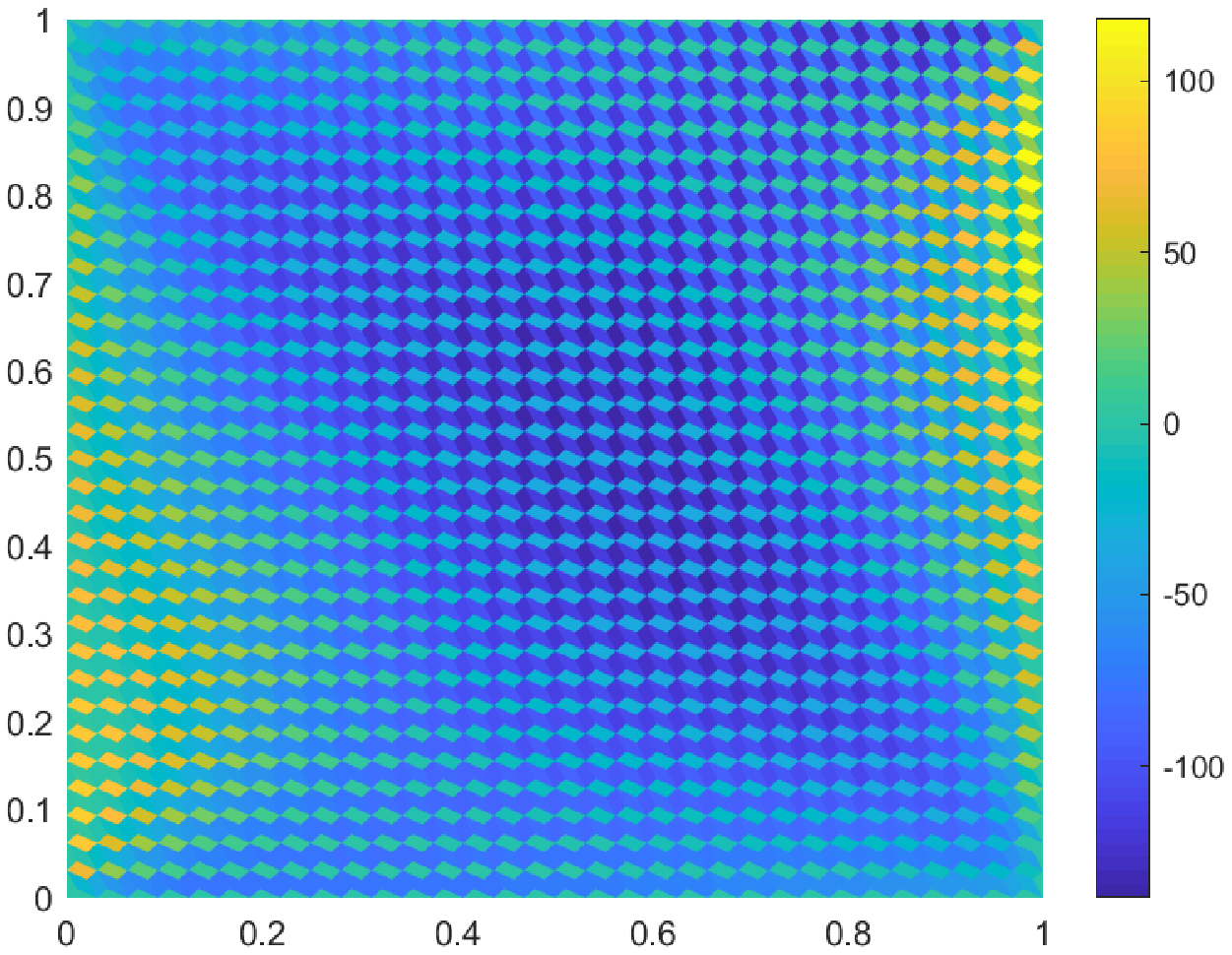}
    \includegraphics[width=0.32\textwidth]{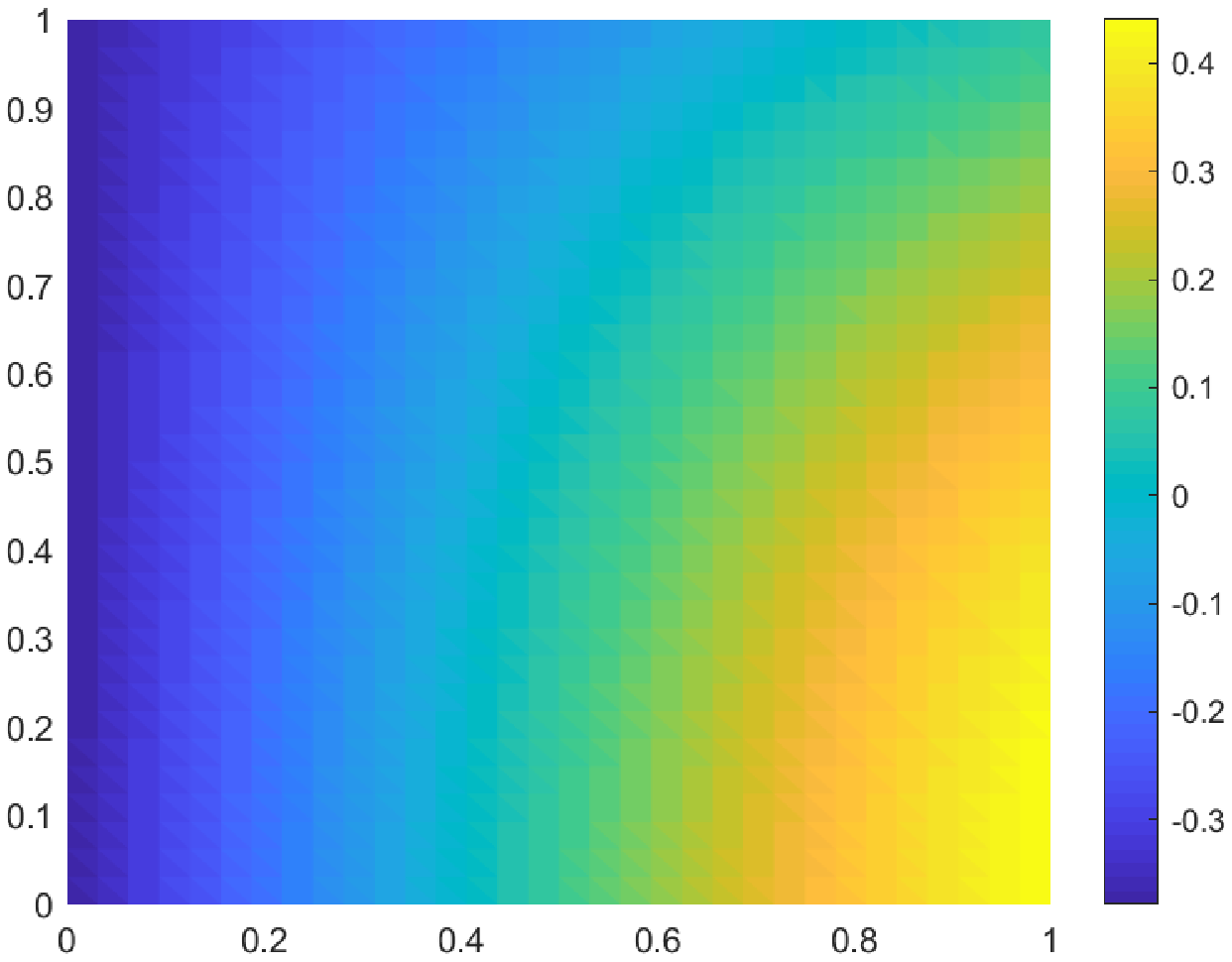}
    \caption{Example~\ref{ex1:first}: Numerical approximations on the mesh with $h=1/32$ with $\nu=10^{-6}$. Top: numerical solution of $u_1$ (left), $u_2$ (middle), and $p$ (right) from SDG1. Bottom: numerical solution of $u_1$ (left), $u_2$ (middle), and $p$ (right) from SDG2.}
    \label{ex1:numerical-sol}
\end{figure}

\subsubsection{No flow}\label{sub:noflow}

In this example, we again set $\Omega=(0,1)^2$, and the exact velocity and pressure are defined by
\begin{align*}
\bm{u}=\left(
         \begin{array}{c}
           0 \\
           0 \\
         \end{array}
       \right),\quad p=-\frac{\text{Ra}}{2}y^2+\text{Ra} y-\frac{\text{Ra}}{3},
\end{align*}
where $\text{Ra}=1000$.

The numerical solution for velocity and pressure for both algorithms are displayed in Figure~\ref{ex1:sol}. We can observe that SDG1 delivers zero velocity fields, which matches the exact solution. However, SDG2 yields nonzero velocity, which is far from the exact velocity. Then we show the error profiles for $\|\bm{u}-\bm{u}_h\|_0$, $\|\bm{\omega}-\bm{\omega}_h\|_0$, $\|p-p_h\|_0$ and $\|\mathcal{I}_h\bm{u}-\bm{u}_h\|_0$. The errors from SDG1 approach zero, see Table~\ref{table1}, whereas this is not the case for SDG2, see Table~\ref{table2}.

\begin{figure}[t]
    \centering
    \includegraphics[width=0.32\textwidth]{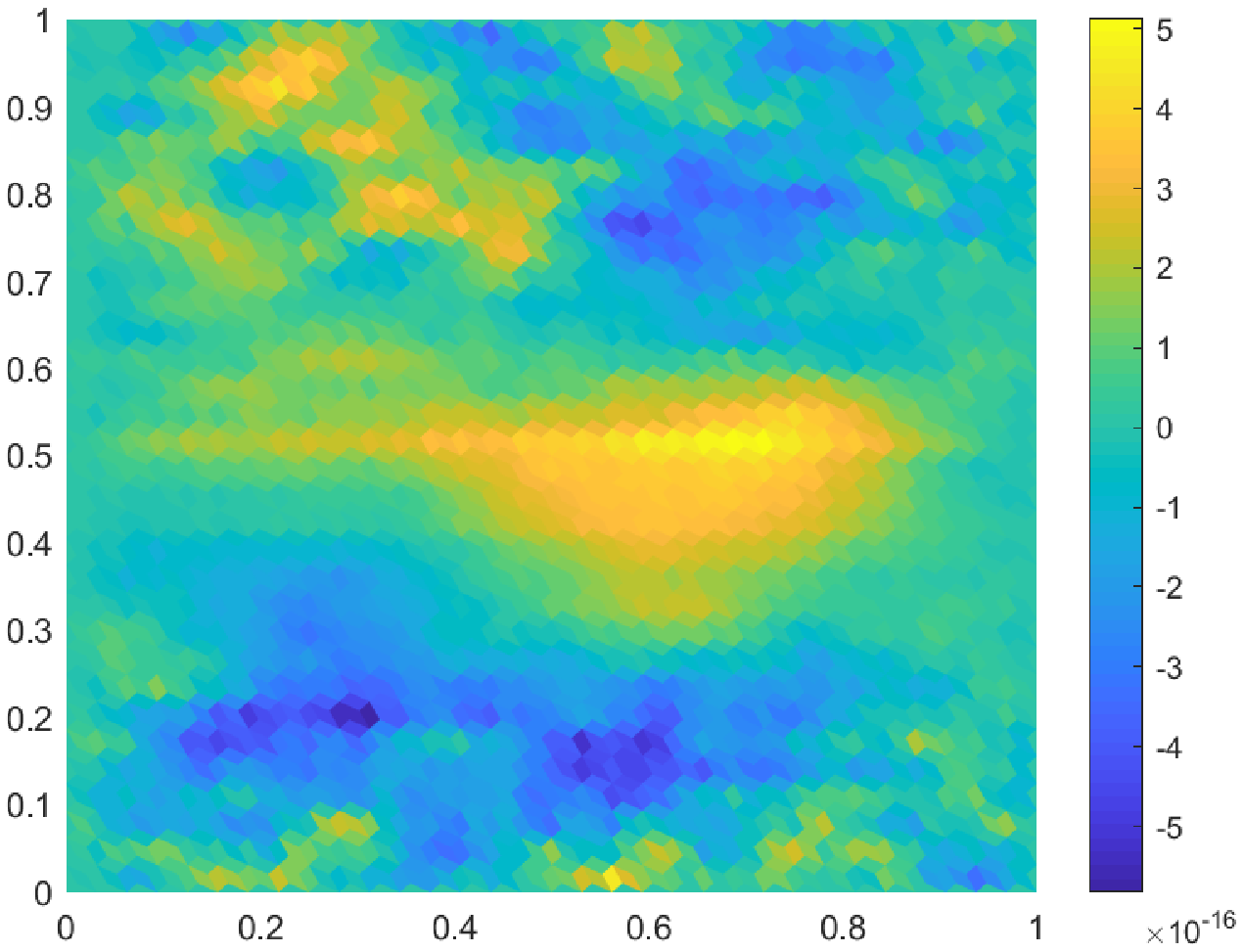}
    \includegraphics[width=0.32\textwidth]{figs/ex1_sol_u1.eps}
    \includegraphics[width=0.32\textwidth]{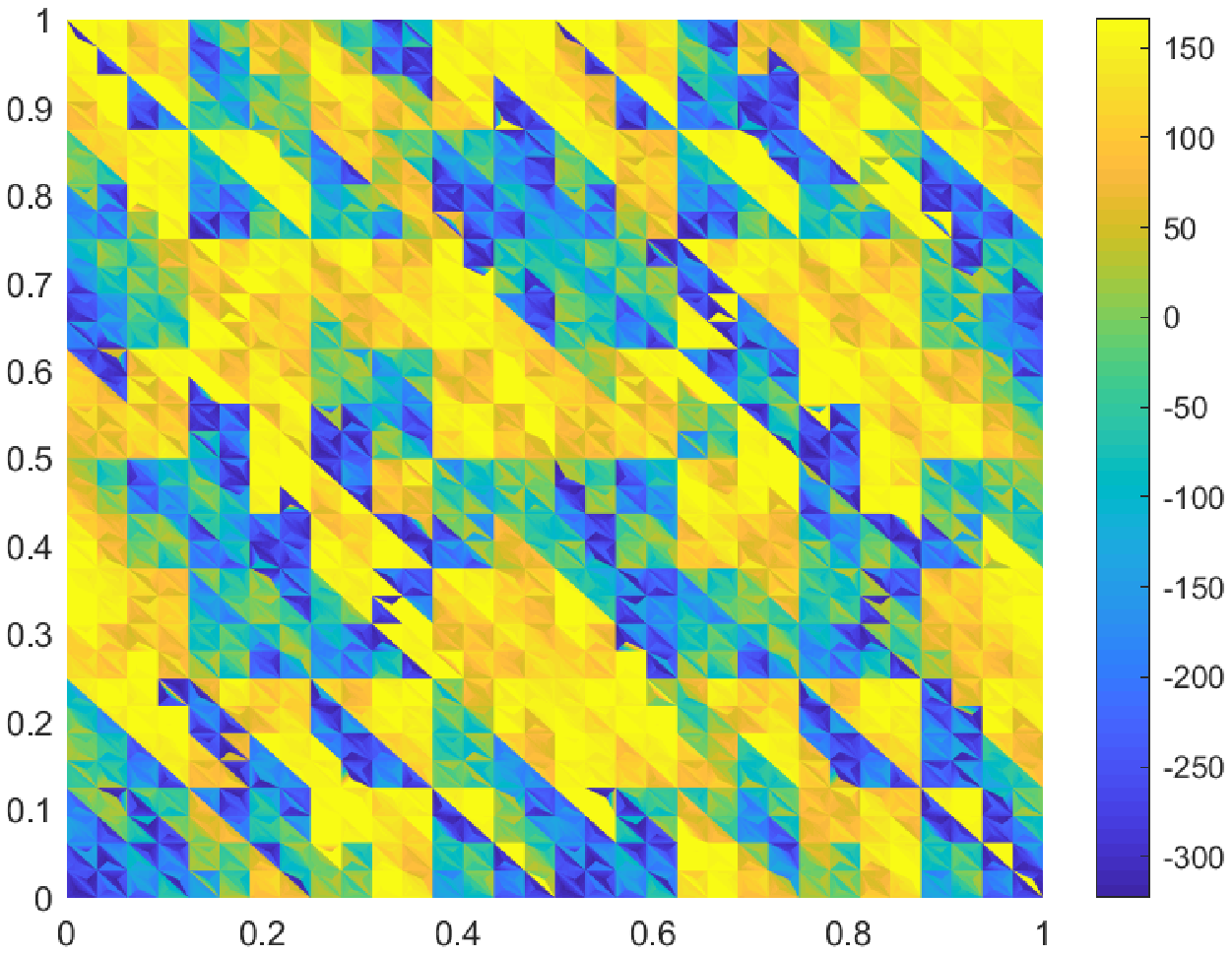}\\
    \includegraphics[width=0.32\textwidth]{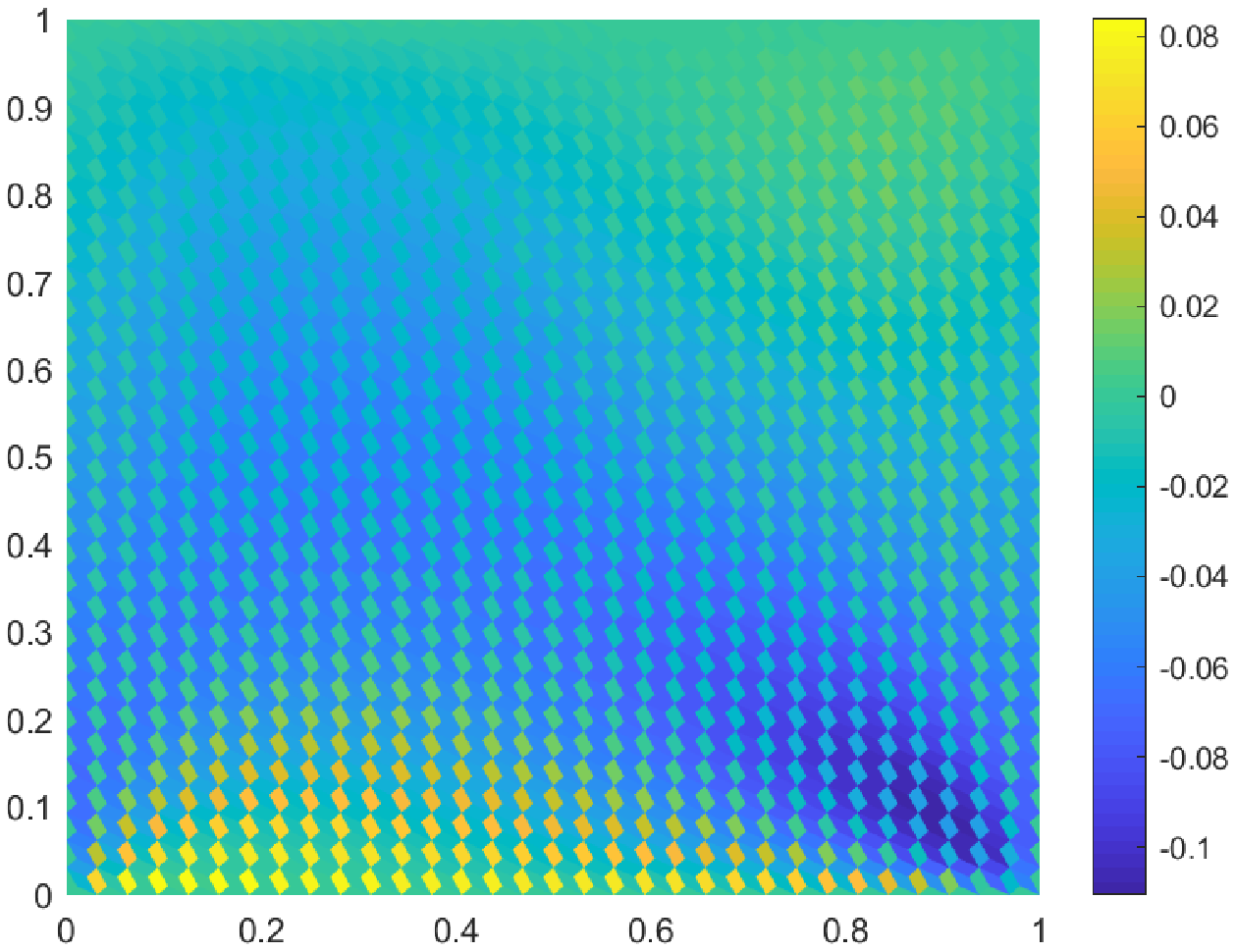}
    \includegraphics[width=0.32\textwidth]{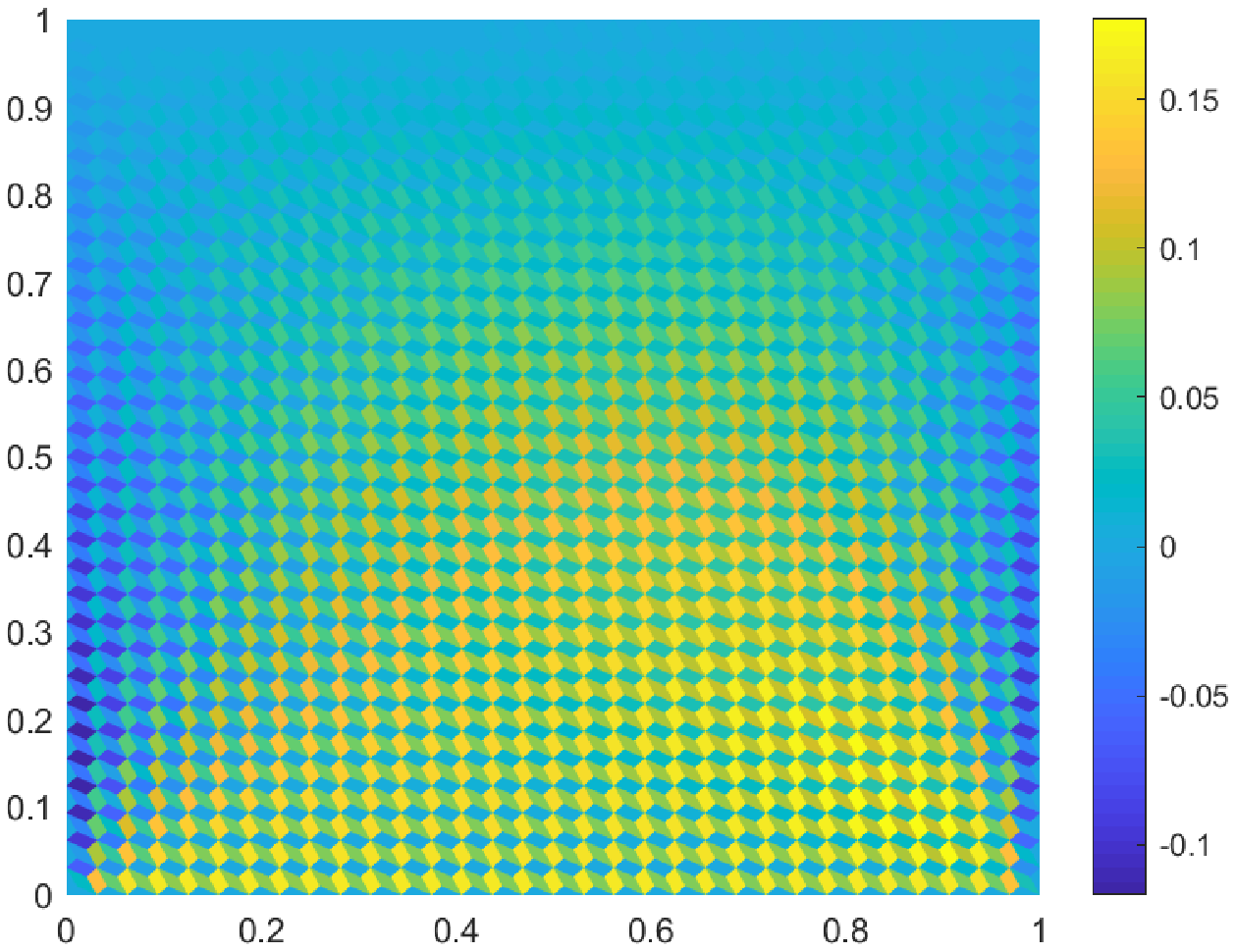}
     \includegraphics[width=0.32\textwidth]{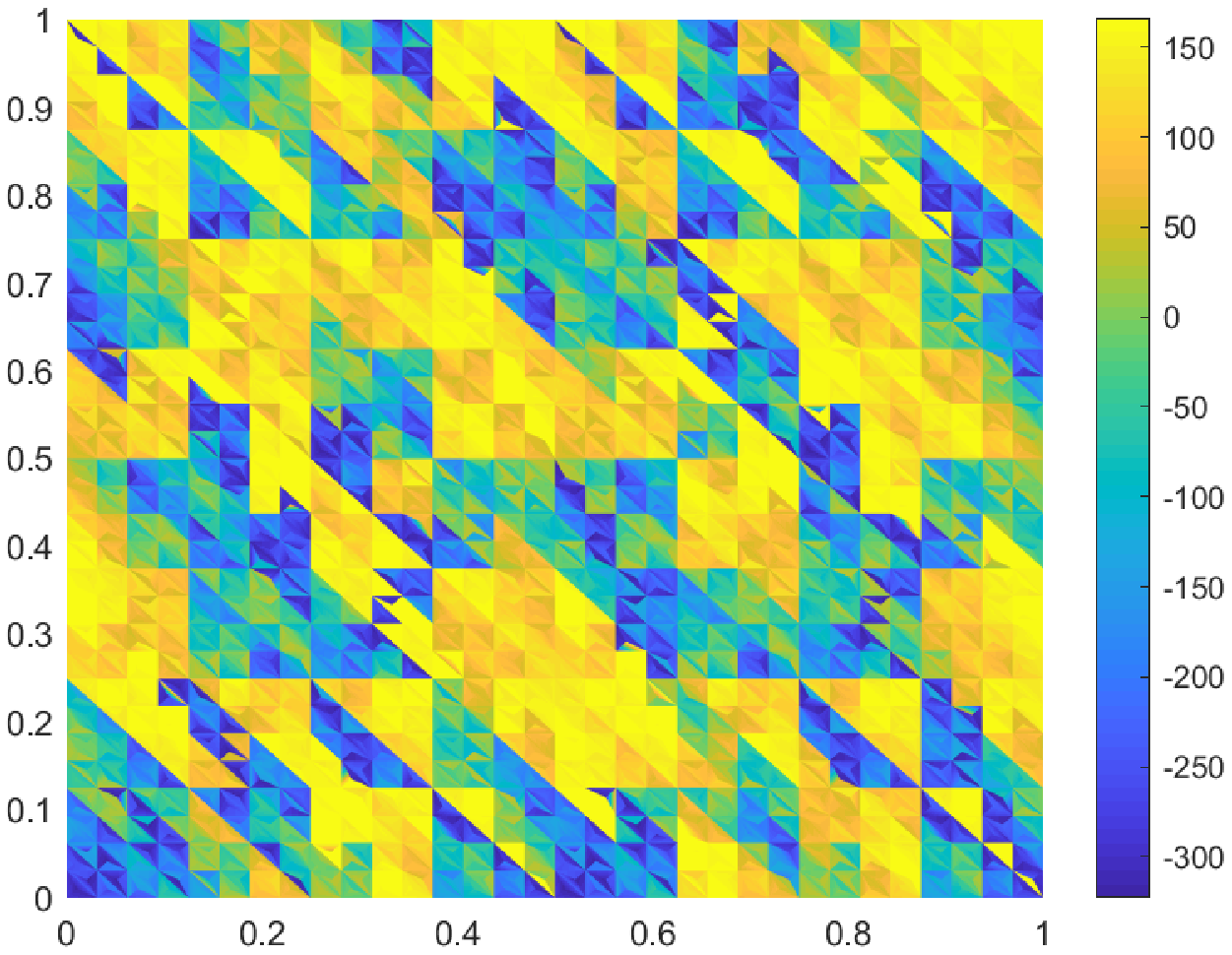}
    \caption{Example~\ref{sub:noflow}: Numerical approximations on the mesh with $h=1/32$. Top: numerical solution of $u_1$ (left), $u_2$ (middle), and $p$ (right) from SDG1. Bottom: numerical solution of $u_1$ (left), $u_2$ (middle), and $p$ (right) from SDG2.}
    \label{ex1:sol}
\end{figure}

\begin{table}[t]
\begin{center}
{\footnotesize
\begin{tabular}{c|c c|c c|c c|c c}
\hline
 Mesh &\multicolumn{2}{|c|}{$\|\bm{u}-\bm{u}_h\|_{0}$} & \multicolumn{2}{|c|}{$\|\bm{\omega}-\bm{\omega}_h\|_{0}$} & \multicolumn{2}{|c}{$\|p-p_h\|_{0}$}& \multicolumn{2}{|c}{$\|\mathcal{I}_h\bm{u}-\bm{u}_h\|_{0}$}\\
\hline
  $h^{-1}$ & Error & Order & Error & Order &  Error & Order & Error & Order\\
\hline
  2  & 1.94e-015 &   N/A    &1.36e-014 &  N/A &66 &   N/A & 1.94e-15& N/A\\
  4  & 5.34e-016 &   1.86   &5.00e-015 & 1.44 &33 & 0.97  & 5.35e-16&1.86 \\
  8  & 3.95e-016 &   0.43   &5.76e-015 & -0.20 &16 & 0.99 & 3.95e-16&0.43\\
  16 & 3.73e-016 &  0.08    &7.42e-015 & -0.36 &8.4 & 0.99&  3.73e-16&0.08\\
  32 & 2.63e-016 &  0.50   &7.14e-015 & 0.05 &4.2 & 0.99 &  2.63e-16&0.50\\
\hline
\end{tabular}}
\caption{Convergence history for SDG1 for Example~\ref{sub:noflow}.}
\label{table1}
\end{center}
\end{table}

\begin{table}[t]
\begin{center}
{\footnotesize
\begin{tabular}{c|c c|c c|c c|c c}
\hline
 Mesh &\multicolumn{2}{|c|}{$\|\bm{u}-\bm{u}_h\|_{0}$} & \multicolumn{2}{|c|}{$\|\bm{\omega}-\bm{\omega}_h\|_{0}$} & \multicolumn{2}{|c}{$\|p-p_h\|_{0}$}& \multicolumn{2}{|c}{$\|\mathcal{I}_h\bm{u}-\bm{u}_h\|_{0}$}\\
\hline
  $h^{-1}$ & Error & Order & Error & Order &  Error & Order & Error & Order\\
\hline
  2  & 8.75 &   N/A    &51 &  N/A &75 &   N/A & 8.75& N/A\\
  4  & 3.60 &   1.27   &34 & 0.57 &38 & 0.95  & 3.60&1.27 \\
  8  & 1.12 &   1.68   &19 & 0.83 &18 & 1.05 & 1.12&1.68\\
  16 & 0.30 &  1.88    &10 & 0.93 &8.9 &1.06&  0.30&1.88\\
  32 & 0.08 &  1.95   &5 & 0.97&4.3 & 1.04 & 0.07&1.96\\
\hline
\end{tabular}}
\caption{Convergence history for SDG2 for Example~\ref{sub:noflow}.}
\label{table2}
\end{center}
\end{table}

\subsection{Trapezoidal mesh}\label{ex:trape}

Let $\Omega=(0,1)^2$ and we choose the exact solution to be
\begin{align*}
\bm{u}=
\left(
  \begin{array}{c}
    -e^{x}(y\cos(y)+\sin(y)) \\
    e^{x}y\sin(y) \\
  \end{array}
\right),\quad p=2e^{x}\sin(y).
\end{align*}

In this test, we employ the trapezoidal mesh shown in Figure~\ref{mesh}. The convergence history against the number of degrees of freedom for $\nu=1$ is reported in Figure~\ref{ex2:accuracy}, and optimal convergence rates matching the theoretical results can be obtained. Moreover, we also show the convergence history against the number of degrees of freedom for $\nu=10^{-6}$ for SDG1, as expected, optimal convergence rates can be obtained. To verify the robustness, we show the errors for various values of $\nu$, i.e., $\nu=10^j (j=-6,\cdots,2)$ for both algorithms with the number of degrees of freedom (dof) to be $4993$ in Figure~\ref{ex2:error-profile}.
Similarly, we can observe that SDG2 is not pressure robust and SDG1 is pressure robust. The velocity error of SDG1 remains a constant for various values of $\nu$ while the velocity error of SDG2 is asymptotically proportional to $1/\nu$ when $\nu\leq 1$. Moreover, the $L^2$ error of velocity gradient from SDG1 is asymptotically proportional to $\nu$, whereas, the $L^2$ error of velocity gradient from SDG2 tends to be a constant when $\nu\leq 1$.

\begin{figure}[t]
    \centering
    \includegraphics[width=0.35\textwidth]{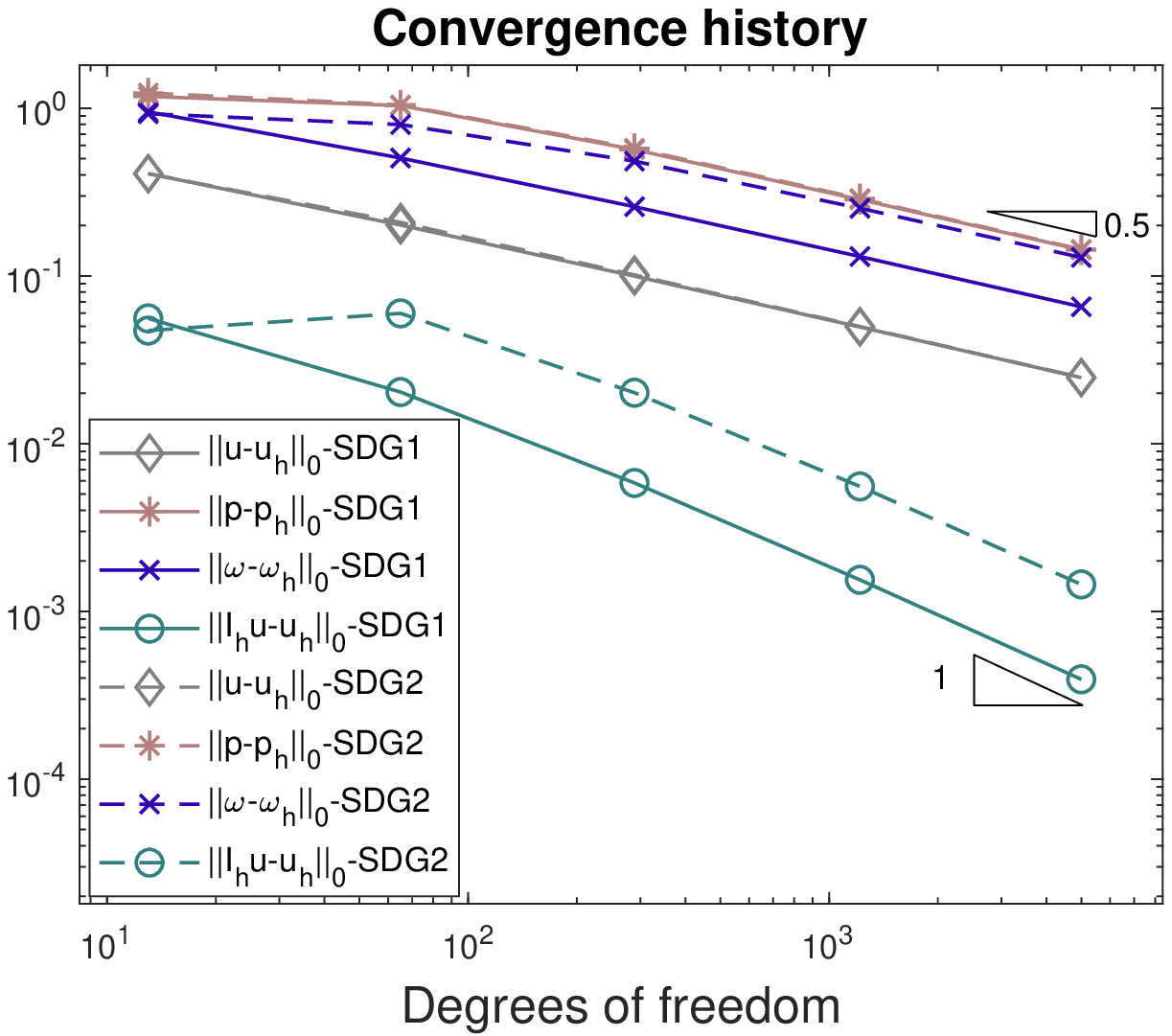}
    \includegraphics[width=0.35\textwidth]{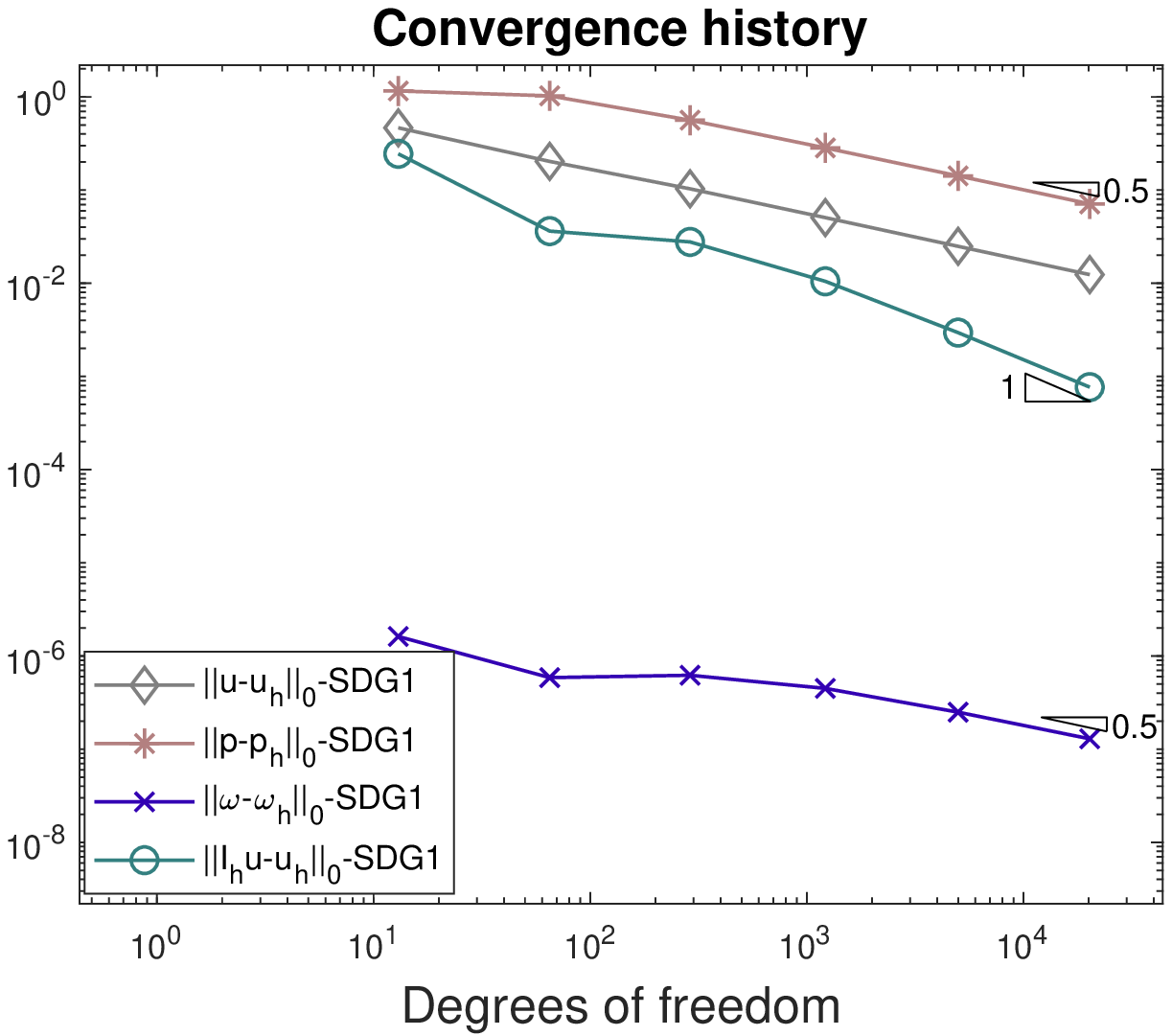}
    \caption{Example~\ref{ex:trape}: Convergence history for $\nu=1$ (left) and $\nu=10^{-6}$ (right).}
    \label{ex2:accuracy}
\end{figure}

\begin{figure}[t]
    \centering
    \includegraphics[width=0.32\textwidth]{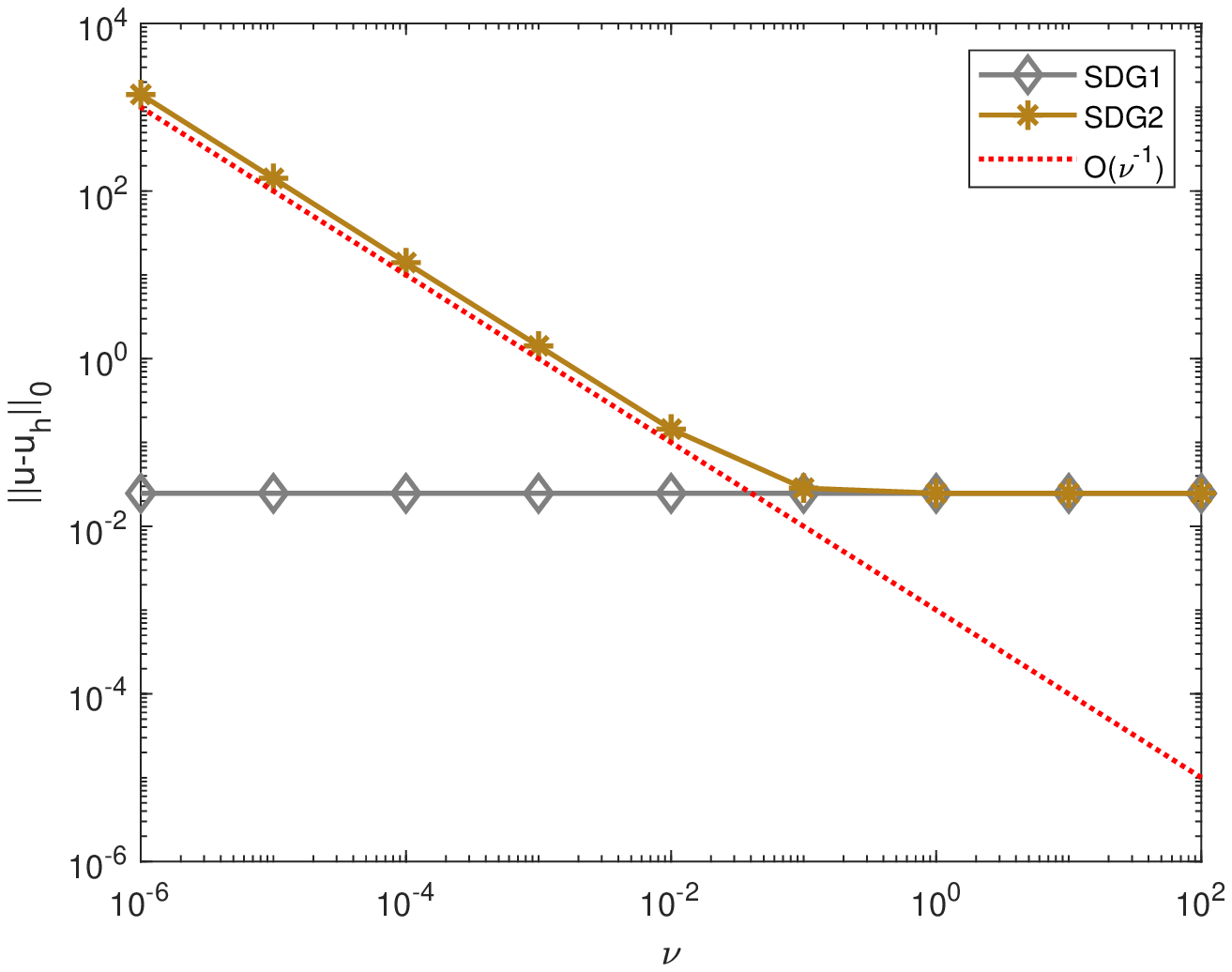}
    \includegraphics[width=0.32\textwidth]{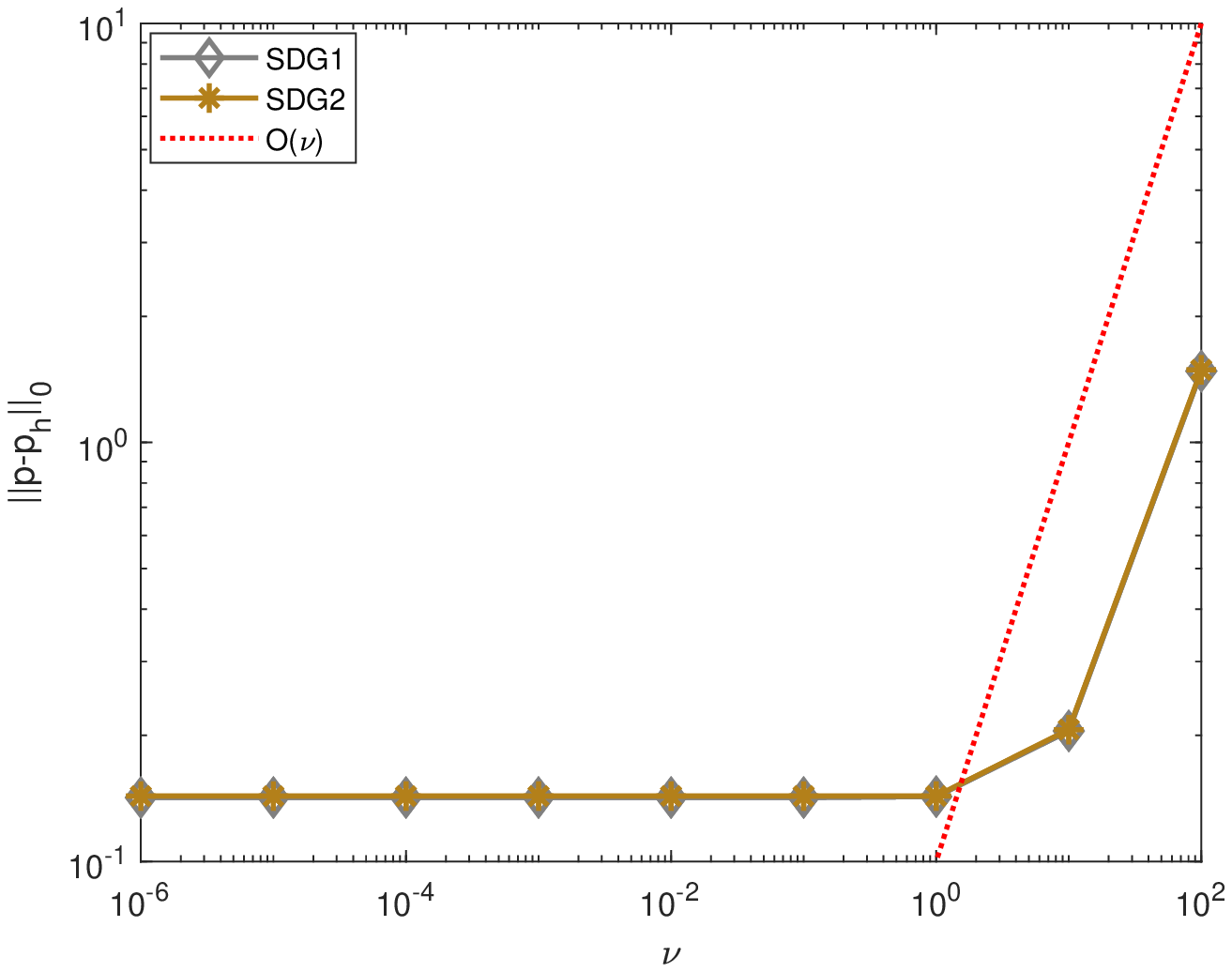}
     \includegraphics[width=0.32\textwidth]{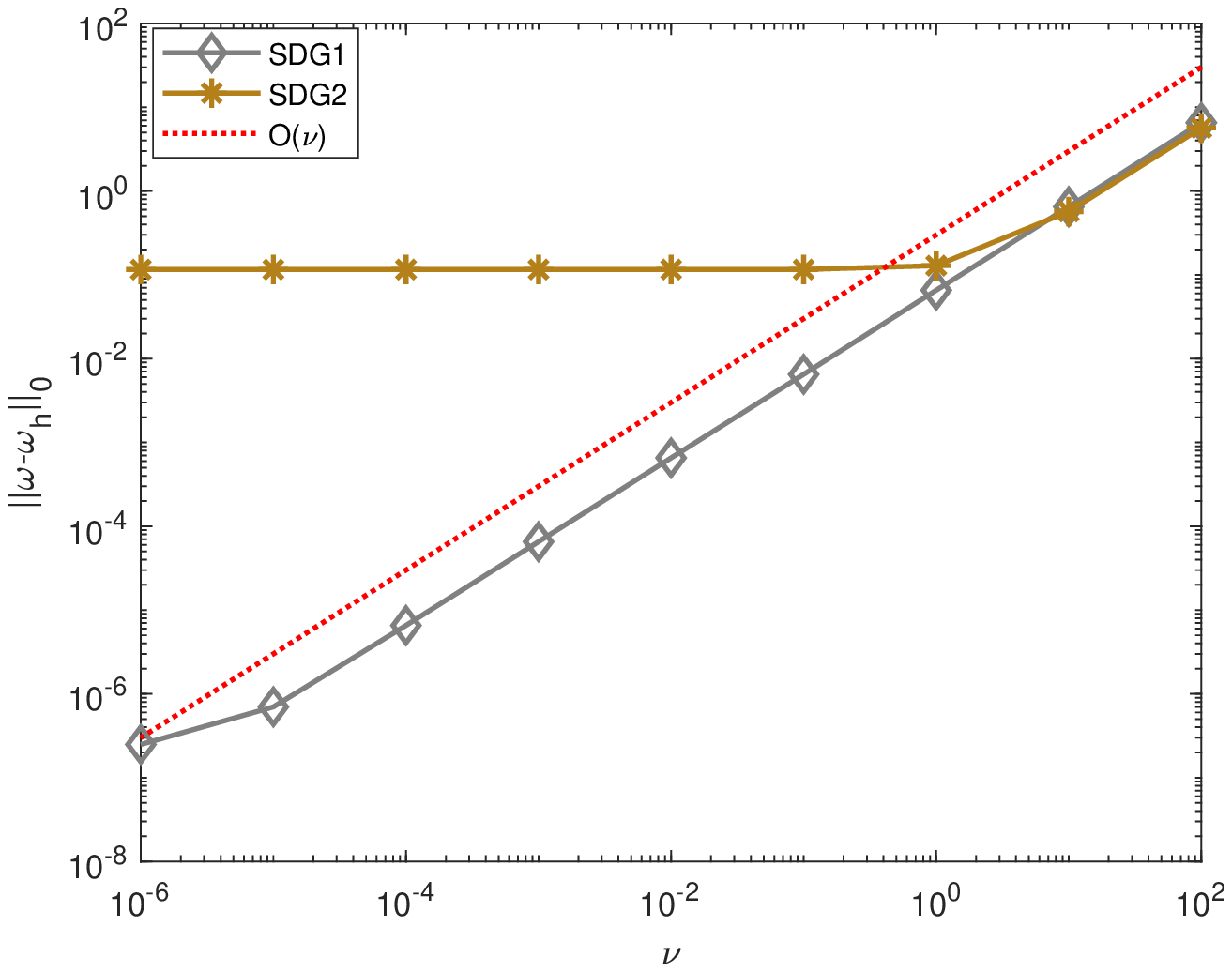}
    \caption{Example~\ref{ex:trape}: Error profiles for velocity (left), pressure (middle) and velocity gradient (right) on trapezoidal mesh with $\text{dof}=4993$.}
    \label{ex2:error-profile}
\end{figure}

\subsection{Polygonal mesh}\label{ex:polygon}
In this example, we choose the same exact solution given in \eqref{eq:exact1} and exploit the polygonal mesh displayed in Figure~\ref{mesh}. Figure~\ref{fig:con} shows the convergence history against the number of degrees of freedom for $\nu=1$. To verify the robustness of our method, we consider $L^2$ errors of velocity, pressure and velocity gradient with various values of $\nu$, i.e., $\nu=10^j (j=-5,\cdots,2)$ on the fixed mesh, and the numerical results are given in Figure~\ref{ex3:accuracy}. Again, we can obtain similar conclusions as sections~\ref{ex:tri} and \ref{ex:trape}.

\begin{figure}[t]
    \centering
    \includegraphics[width=0.35\textwidth]{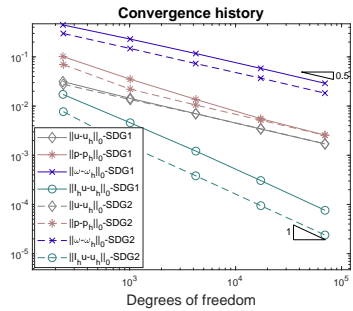}
    \caption{Example~\ref{ex:polygon}: Convergence history for $\nu=1$.}
    \label{fig:con}
\end{figure}

\begin{figure}[t]
    \centering
    \includegraphics[width=0.32\textwidth]{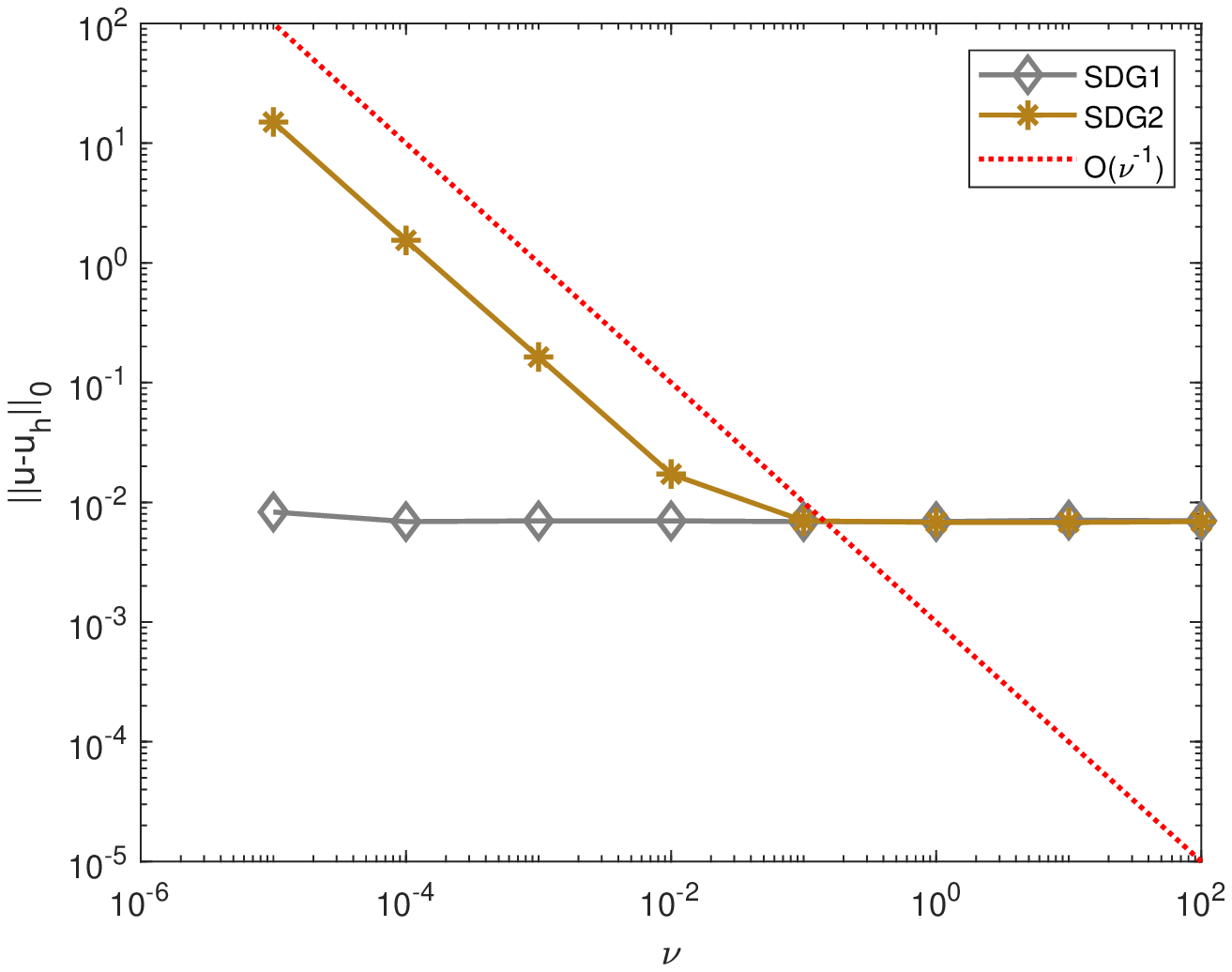}
    \includegraphics[width=0.32\textwidth]{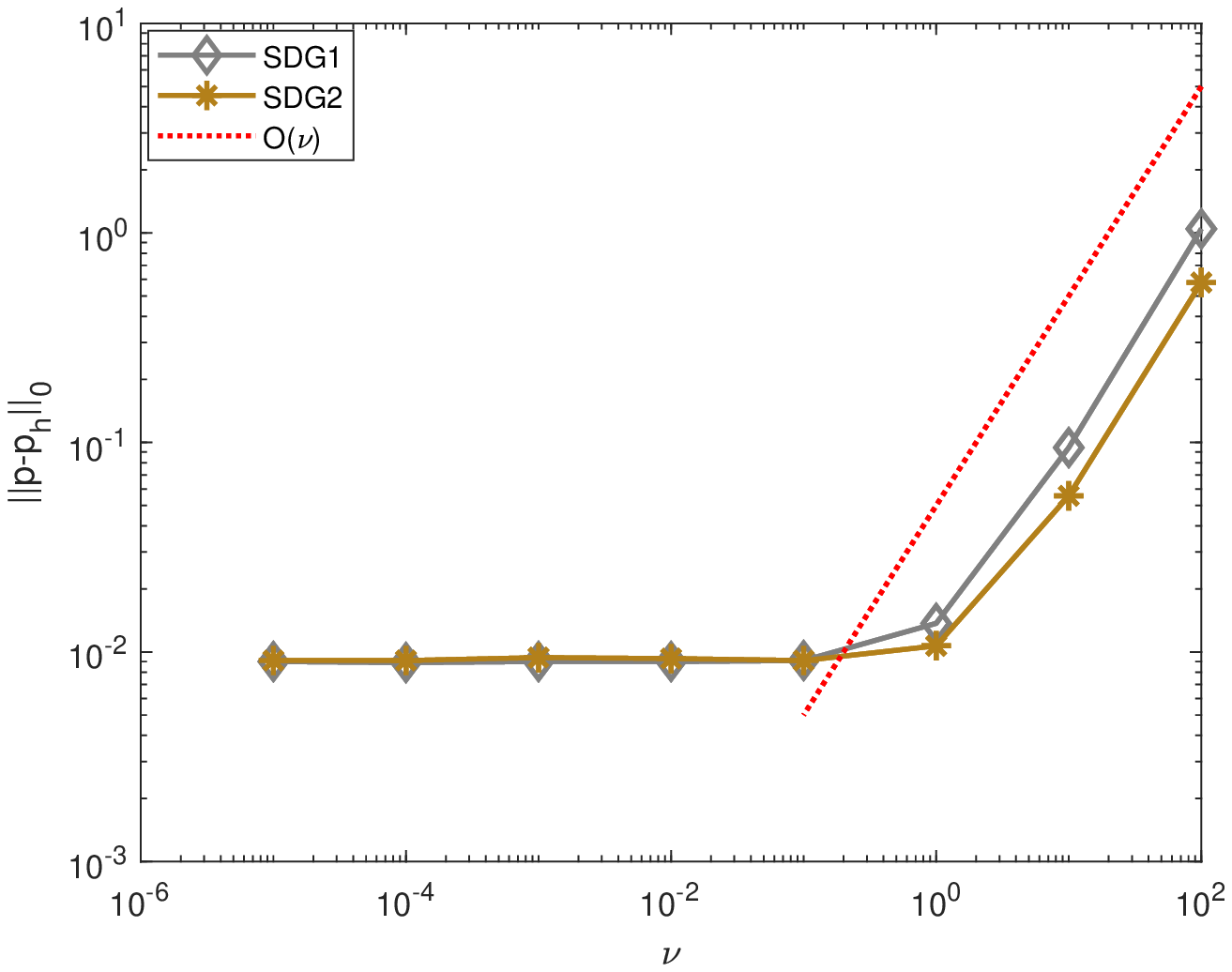}
     \includegraphics[width=0.32\textwidth]{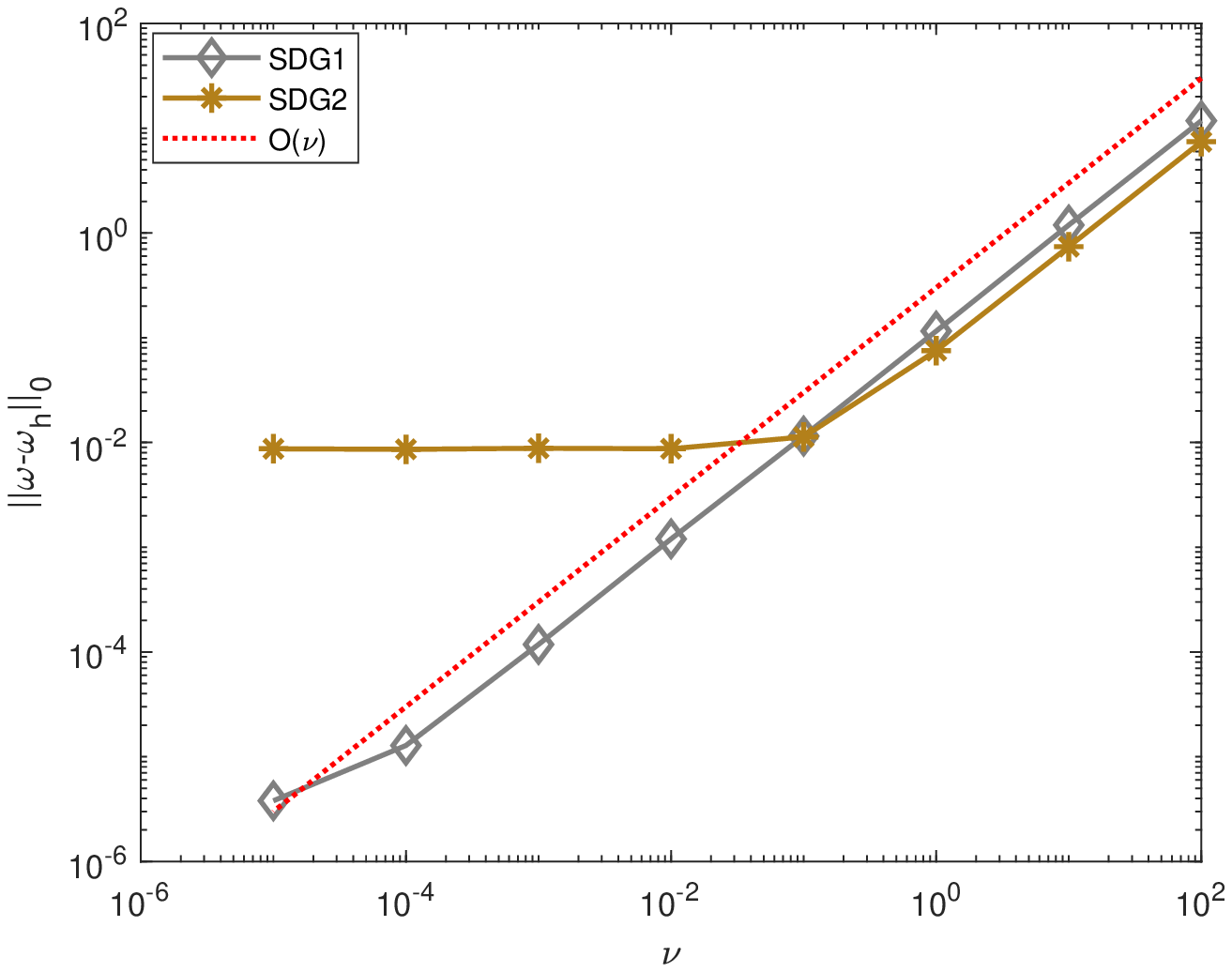}
    \caption{Example~\ref{ex:polygon}: Error profiles for velocity (left), pressure (middle) and velocity gradient (right) on polygonal mesh with $\text{dof}=4235$.}
    \label{ex3:accuracy}
\end{figure}

%

\section{Conclusion}
In this paper we have developed a pressure robust staggered discontinuous Galerkin method for the Stokes equations, where the crux is to modify the right hand side by using divergence preserving operator. It is proved theoretically that the velocity error estimates are independent of $\nu$. In addition, we are able to show that the numerical approximation for velocity superconverges to a suitable projection. Several numerical experiments are carried out to test the accuracy and robustness of the proposed method. In the future we will extend this approach to solve Navier-Stokes equations with arbitrary polynomial orders, in which case the smallness assumption given in \cite{ChungQiu17} can be weakened to concern only the solenoidal part $\bm{g}$ of the body force.

\section*{Acknowledgments}

The research of Eric Chung is partially supported by the Hong Kong RGC General Research Fund (Project numbers 14304217 and 14302018) and CUHK Faculty of Science Direct Grant 2019-20. The research of Eun-Jae Park is supported by the National Research Foundation of Korea (NRF) grant funded by the Ministry of
Science and ICT (NRF-2015R1A5A1009350 and NRF-2019R1A2C2090021).


\end{document}